\documentclass[a4paper,11pt,pdftex]{amsart}

\usepackage{amsmath}
\usepackage{amssymb}
\usepackage{ascmac}
\usepackage{amsthm}
\usepackage{color}
\usepackage{enumerate}
\usepackage{enumitem}
\setlist[enumerate,1]{label={\normalfont(\roman*)}}
\usepackage{mathrsfs}
\usepackage{graphicx}
\usepackage{braket}
\usepackage{wrapfig}
\usepackage{layout}
\usepackage{cite}
\usepackage[margin=1.9cm, bottom=3cm, footskip=1.5cm]{geometry}
\usepackage{layout}

\numberwithin{equation}{section}
\usepackage
[draft=false,setpagesize=false,pdfstartview=FitH,
colorlinks=true,linkcolor=blue,citecolor=magenta,pagebackref=true,bookmarks=false]
{hyperref}

\usepackage{float}

\newcommand{\R}{\mathbf{R}}
\newcommand{\Z}{\mathbf{Z}}

\newcommand{\RN}{\mathbf{R}^N}

\newcommand{\e}{\varepsilon}
\newcommand{\be}{\bar{\e}}
\newcommand{\rd}{\mathrm{d}}

\newcommand{\cD}{\mathcal{D}}
\newcommand{\cL}{\mathcal{L}}
\newcommand{\cO}{\mathcal{O}}

\newcommand{\scrF}{\mathscr{F}}

\newcommand{\la}{\left\langle}
\newcommand{\ra}{\right\rangle}
\newcommand{\RE}{\mathrm{Re}\,}
\newcommand{\PV}{\mathrm{P.V.}}
\newcommand{\cDsr}{\mathcal{D}^s_{\rm rad}(\RN)}
\newcommand{\Hsr}{H^s_{\rm rad}}
\def\ov#1{\overline{#1}}

\def\wt#1{\widetilde{#1}}

\def\cDs#1{\mathcal{D}^s(#1)} 
\newcommand{\PSPc}{${\rm (PSP)}_c$}

\theoremstyle{definition}
\newtheorem{definition}{Definition}[section]
\theoremstyle{plain}
\newtheorem{theorem}[definition]{Theorem}

\newtheorem{lemma}[definition]{Lemma}
\newtheorem{proposition}[definition]{Proposition}
\theoremstyle{remark}
\newtheorem{remark}[definition]{Remark}

\title[]{Multiplicity of radial and nonradial solutions to equations with fractional operators}
\author[]{Norihisa Ikoma}
\date{\today}
\address{Department of Mathematics,
	Faculty of Science and Technology,
	Keio University, 
	Yagami Campus: 3-14-1 Hiyoshi, Kohoku-ku, 
	Yokohama, 2238522, JAPAN}
\email{ikoma@math.keio.ac.jp}

\begin{document}
\maketitle
\begin{abstract}
In this paper, we study the existence of radial and nonradial solutions 
to the scalar field equations with fractional operators. 
For radial solutions, we prove the existence of infinitely many solutions under $N \geq 2$. 
We also show the existence of least energy solution (with the Pohozaev identity) 
and its mountain pass characterization. 
For nonradial solutions, we prove the existence of at least one nonradial solution 
under $N \geq 4$ and infinitely many nonradial solutions under either $N =4$ or $N \geq 6$. 
We treat both of the zero mass and the positive mass cases. 
\end{abstract}

	\section{Introduction}\label{section:1}


In this paper, we discuss the existence of radial and nonradial solutions of 
	\begin{equation}\label{eq:1.1}
		(-\Delta)^s u = f(u) \quad {\rm in} \ \RN, \quad 
		u \in \cDs{\RN}
	\end{equation}
and 
	\begin{equation}\label{eq:1.2}
		(-\Delta + a )^s u = f(u) \quad {\rm in} \ \RN, \quad u \in H^s(\RN).
	\end{equation}
Here  $N \geq 2$, $0<s<1$ and $a \geq 0$, and 
the fractional operators are defined by
	\[
		\left( - \Delta + a \right)^s u (x) := 
		\scrF^{-1} \left( \left( 4\pi^2|\xi|^2 + a \right)^s \scrF u (\xi) \right), 
		\quad 
		\scrF u (\xi) := \int_{  \RN } e^{-2\pi i x \cdot \xi} u(x) \rd x.
	\]
For \eqref{eq:1.1}, we deal with the zero mass case, namely, 
we impose the conditions (f0)--(f3) below on the nonlinearity $f(t)$. 
On the other hand, for \eqref{eq:1.2}, we consider the positive mass case, 
that is, the nonlinearity $f(t)$ satisfying (F0)--(F3).

	Next, we give some remarks about the fractional operators and 
the function spaces. First, for Schwartz functions $u$ and $v$, we may check that 
(see Di Nezza, Palatucci and Valdinoci \cite{DNGE-12}, for instance)
	\[
		\begin{aligned}
			(-\Delta)^s u (x) &= C_{N,s} \PV \int_{\RN} 
			\frac{u(x)- u(y)}{|x-y|^{N+2s} } \rd y 
			= -\frac{C_{N,s}}{2} \int_{\RN} \frac{u(x+y)+u(x-y) - 2u(y) }{|y|^{N+2s}} \rd y
			,\\
			\int_{  \RN } \left[\left( - \Delta \right)^s u \right] v \rd x 
			&=
			\frac{C_{N,s}}{2}\int_{\RN} \frac{\left( u(x) - u(y) \right) \left( v(x) - v(y) \right) }{|x-y|^{N+2s}} \rd x \rd y 
			= : \la u, v \ra_{\cDs{\RN}},
			\\
			\int_{  \RN } \left[ \left( -\Delta + a  \right)^s u \right] v \rd x 
			&= \RE \int_{\RN} \left( 4\pi^2 |\xi|^2 + a \right)^s 
			\scrF u (\xi) \ov{\scrF v (\xi)} \rd \xi 
			=: \la u , v \ra_{s,a}.
		\end{aligned}
	\]
where 
	\[
		C_{N,s}^{-1} := \int_{\RN} \frac{1 - \cos (y_1)}{|y|^{N+2s}} \rd y.
	\]
Now we define $\cDs{\RN}$ and $H^s(\RN)$ by
	\[
		\begin{aligned}
			\cDs{\RN} &:= 
			\Set{ u \in L^{2^\ast_s} (\RN) 
				|
			\| u \|_{\cDs{\RN}}^2 := \la u , u \ra_{\cDs{\RN}}  < \infty } 
		\quad \text{where} \ 2^\ast_s := \frac{2N}{N-2s}, 
			\\
			H^s(\RN) &:= 
			\Set{ u \in L^2(\RN) 
				|
			\| u \|_{s,a}^2 := \la u, u \ra_{s,a}^2 < \infty.
		 }
		\end{aligned}
	\]
Remark that due to Sobolev's inequality, 
$\cD^s(\RN)$ and $H^s(\RN)$ are Hilbert spaces 
and $C^\infty_0(\RN)$ is dense in $\cD^s(\RN)$ and $H^s(\RN)$.

	Let us introduce the conditions for $f(t)$ in \eqref{eq:1.1} and \eqref{eq:1.2}. 
For \eqref{eq:1.1}, we deal with the zero mass case and suppose
	\begin{enumerate}
	\item[(f0)] 
		$f \in C(\R)$ and $f(-t) = - f(t)$ for $t \in \R$. 
	\item[(f1)] 
		\[
			- \infty \leq 
			\limsup_{t \to 0} \frac{f(t)}{|t|^{2^\ast_s - 2}t } \leq 0.
		\]
	\item[(f2)] 
		\[
			- \infty \leq \limsup_{t \to \infty} \frac{f(t)}{|t|^{2^\ast_s - 2}t} \leq 0.
		\]
	\item[(f3)] 
		There exists a $\zeta_1>0$ such that $F(\zeta_1) > 0$ 
		where $F(t) := \int_0^t f(\tau) \rd \tau$. 
\end{enumerate}
A simple example satisfying (f0)-(f3) is 
$f(t) = - |t|^{p-1} t + \sum_{i=1}^k a_i |t|^{p_i-1}t + |t|^{q-1} t$ where 
$0<p < p_1 < p_2 < \cdots < p_k <q<2^\ast_s - 1$ and $a_i \in \R$ $(1 \leq i \leq k)$.

	For \eqref{eq:1.2}, assume 
\begin{enumerate}
	\item[(F0)] 
		$f \in C(\R)$ and $f(-t) = - f(t)$ for $t \in \R$. 
	\item[(F1)] 
		\[
			- \infty < \liminf_{t \to 0} \frac{f(t)}{t} \leq \limsup_{t \to 0} \frac{f(t)}{t} < a^s .
		\]
	\item[(F2)] 
		\[
			- \infty \leq \limsup_{t \to \infty} \frac{f(t)}{|t|^{2^\ast_s-2} t } \leq 0.
		\]
	\item[(F3)] 
		There exists a $\zeta_0>0$ such that $F(\zeta_0) - \frac{a^s}{2} \zeta_0^2 > 0$. 
\end{enumerate}
An example satisfying (F0)--(F3) is 
$f(t) = - t + \sum_{i=1}^k a_i |t|^{p_i-1} t + |t|^{q-1}$ where $1 < p_1 < p_2 < \cdots < p_k < q < 2^\ast_s-1$ 
and $a_i \in \R$ ($1 \leq i \leq k$).

	Before proceeding, we give one remark about the notion of solutions of \eqref{eq:1.1} and \eqref{eq:1.2}. 
Formally, \eqref{eq:1.1} is variational, namely, a critical point of 
	\[
		I_0(u) := \frac{1}{2} \| u \|_{\cDs{\RN}}^2 - \int_{\RN} F(u) \rd x
	\]
corresponds to solutions of \eqref{eq:1.1}. However, (f0)--(f3) do not ensure that 
$F(u) , f(u) v \in L^1(\RN)$ for every $u,v \in \cDs{\RN}$, 
hence, the case $I_0 \not\in C^1(\cDs{\RN},\R)$ may occur. 
Therefore, in this paper, solutions of \eqref{eq:1.1} mean weak solutions, 
that is, $u \in \cDs{\RN}$ satisfying 
	\[
		\la u, \varphi \ra_{\cDs{\RN}} = \int_{  \RN } f(u) \varphi \rd x \quad 
		\text{for all } \varphi \in C^\infty_0(\RN). 
	\]

	A similar argument is applied to solutions of \eqref{eq:1.2} 
since (F0)--(F3) may not imply $I_a \in C^1(H^s(\RN), \R)$ where 
	\[
		 I_a (u) := \frac{1}{2} \| u \|_{s,a}^2 
		- \int_{  \RN } F(u) \rd x.
	\]
Thus, we look for weak solutions of \eqref{eq:1.2}, namely, $u \in H^s(\RN)$ satisfying 
	\[
		\la u, \varphi \ra_{s,a} = \int_{  \RN } f(u) \varphi \rd x \quad 
		\text{for each } \varphi \in C^\infty_0(\RN) .
	\]

	Next, we mention known results related to \eqref{eq:1.1} and \eqref{eq:1.2}. 
For \eqref{eq:1.1}, we first refer to \cite{BL-83-1,BL-83-3,Stru-82,Stru-86}. In these papers, 
\eqref{eq:1.1} with $s=1$ is studied and it is shown that under (f0)--(f3) with $s=1$ and $N \geq 3$, 
\eqref{eq:1.1} has infinitely many radial solutions. 
For the existence of nonradial solutions with $s=1$, recently, Mederski \cite{M-19} proved the existence of 
at least one nonradial solution when $N \geq 4$ and 
infinitely many solutions when either $N=4$ or else $N \geq 6$. 
For the case $0<s<1$, Ambrosio \cite{Amb-18-1,Amb-18-2} showed the existence of at least one radial solution 
by assuming $N \geq 2$ and some additional conditions on $f(t)$ to (f0)--(f3).

	On the other hand, for \eqref{eq:1.2}, 
in the case $s=1$, the existence of least energy solution and infinitely many radial solutions 
were proved in \cite{BGK-83,BL-83-1,BL-83-2,S-77} under (F0)--(F3) with $s=1$. See also \cite{HIT-10,JL-18}. 
For the nonradial solutions, we first mention the result by Bartsch and Willem \cite{BW-93} in which 
they proved the existence of infinitely many nonradial solutions under $N=4$ or $N \geq 6$. 
Recently, Mederski \cite{M-17} generalized \cite{BW-93} and showed the existence of at least one 
nonradial solution for $N \geq 4$ and infinitely many solutions for $N=4$ or $N \geq 6$. 
The same result to \cite{M-17} was also proved in Jeanjean and Lu \cite{JL-18} 
via the monotonicity trick with the symmetric mountain pass theorem. 
For $ 0< s < 1$, when $a = 0$ and $a>0$, the existence of infinitely many radial solutions were proved 
in \cite{Amb-18-1,I-17,I-17-2}. We also refer to the earlier results 
\cite{CW-13,FQT-12,FV-15,Se,TWY-12} and references therein.

	Motivated by the above works, we aim to address the following questions in this paper:
	\begin{itemize}
		\item 
			the existence of infinitely many radial solutions, least energy solution 
			and  nonradial solutions of \eqref{eq:1.1}. 
			This is a fractional counterpart of \cite{BL-83-3,Stru-82,Stru-86,M-19} and 
			extends the result by \cite{Amb-18-1,Amb-18-2}. 
		\item 
			the existence of nonradial solutions for \eqref{eq:1.2}. 
			This question corresponds to \cite{M-17,JL-18}. 
	\end{itemize}

	We first establish the existence of infinitely many radial solutions of \eqref{eq:1.1}:

	\begin{theorem}\label{theorem:1.1}
		Under $N \geq 2$ and \textup{(f0)--(f3)}, 
		\eqref{eq:1.1} has infinitely many radial solutions $(u_k)_k \subset \cDs{\RN}$ 
		which satisfy $|I_0(u_k)| < \infty$ and $ \| u_k \|_{\cDs{\RN}} \to \infty$. 
	\end{theorem}

	\begin{remark}
		When $s=1$, the property $I_0(u_k) \to \infty$ as $k \to \infty$ is shown in \cite{BL-83-3,Stru-82,Stru-86}. 
		Under some additional conditions to (f0)--(f3), we may prove $I_0(u_k) \to \infty$ as $k \to \infty$. 
		See Remark \ref{remark:3.6}. 
	\end{remark}

	Next, we turn to the existence of least energy solution of \eqref{eq:1.1} and its characterization. 
To this aim, we strength (f1) and (f2) as follows:
	\begin{equation}\label{eq:1.3}
		- \infty < \liminf_{t \to 0} \frac{f(t)}{|t|^{2^\ast_s-2}t} \leq 
		\limsup_{t \to 0} \frac{f(t)}{|t|^{2^\ast_s-2}t} \leq 0, \quad 
		- \infty < \liminf_{|t| \to \infty} \frac{f(t)}{|t|^{2^\ast_s-2}t} \leq 
		\limsup_{|t| \to \infty} \frac{f(t)}{|t|^{2^\ast_s-2}t} \leq 0.
	\end{equation}
Notice that \eqref{eq:1.3} implies $I_0 \in C^1( \cDs{\RN} , \R )$. 
	\begin{theorem}\label{theorem:1.3}
		Assume $N \geq 2$, \textup{(f0)}, \eqref{eq:1.3} and  \textup{(f3)}. 
		Then there exists a $u_0 (x) = u_0(|x|) \in \cD^s(\RN)$ such that 
			\begin{equation}\label{eq:1.4}
				0< c_{mp} =  I_0(u) = \inf \left\{ I_0(v) \ |\ 
				\text{$v \in \cD^s(\RN) \setminus \{0\} $ is a solution of \eqref{eq:1.1} with $P_0(v) = 0$
				}
				\right\}
			\end{equation}
		where 
			\[
				\begin{aligned}
					c_{mp} &:= \inf_{ \gamma \in \Gamma} \max_{0 \leq t \leq 1} I_0(\gamma (t)), 
					\quad 
					\Gamma := \Set{ \gamma \in C([0,1], \cDsr ) | \gamma (0) = 0, \ I_0(\gamma (1) ) < 0  },
					\\
					\cDsr &:= \Set{ u \in \cDs{\RN} | u(x) = u(|x|) }, \quad 
					P_0(u) := \frac{N-2s}{2} \| u \|_{\cDs{\RN}}^2 - N \int_{  \RN } F(u) \rd x.
				\end{aligned}
			\]
	\end{theorem}

	\begin{remark}
		(i) The equality $P_0(u) = 0$ corresponds to the Pohozaev identity. 
		
		(ii) \eqref{eq:1.4} is obtained for \eqref{eq:1.2} with $a>0$ in Ikoma \cite{I-17} and 
		this is a fractional counterpart of Jeanjean and Tanaka \cite{JT-03-1}.
	\end{remark}

	Finally, we discuss the existence of nonradial solutions to \eqref{eq:1.1} and \eqref{eq:1.2}. 
To this aim, we follow the setting in \cite{BW-93,M-17,JL-18}. 
Let $N \geq 4$ and set 
	\[
		\cO_1 := O(m_1) \times O(m_1) \times \{ \mathrm{id}_{\R^{N-2m_1}} \} \subset O(N), \quad 
		\cO_2 := O(m_2) \times O(m_2) \times O(N-2m_2) \subset O(N)
	\]
where $2 \leq m_1$ and $2 \leq m_2$ with either $N-2m_2 = 0$ or else $N-2m_2 \geq 2$. 
Remark that in the latter case, we suppose $N=4$ or $N \geq 6$ in addition. 
Writing $x = (x_1,x_2,x_3) \in \R^{m_i} \times \R^{m_i} \times \R^{N-2m_i} = \RN$, 
we define 
\[
\begin{aligned}
\cD_{\cO_1}^s &:= 
\Set{ u \in \cDs{\RN} |
	u(x_1,x_2,x_3) = u(|x_1|,|x_2|,x_3), \ 
	u(x_2,x_1,x_3) = - u(x_1,x_2,x_3)  },
\\ 
\cD_{\cO_2}^s &:= 
\Set{ u \in \cDs{\RN} |
	u(x_1,x_2,x_3) = u(|x_1|,|x_2|,|x_3|), \ 
	u(x_2,x_1,x_3) = - u(x_1,x_2,x_3)  },
\\
H^s_{\cO_1} &:= 
\Set{ u \in H^s(\RN) | 
	u (x_1,x_2,x_3) = u(|x_1|,|x_2|,x_3), \ 
	u(x_2,x_1,x_3) = - u(x_1,x_2,x_3)},
\\
H^s_{\cO_2} &:= 
\Set{ u \in H^s(\RN) | 
	u (x_1,x_2,x_3) = u(|x_1|,|x_2|,|x_3|), \ 
	u(x_2,x_1,x_3) = - u(x_1,x_2,x_3)}.
\end{aligned}
\]
Remark that $u(gx) = u(x)$ for any $g \in \cO_i$ and $u \in \cD^s_{\cO_i}$, 
and due to the antisymmetry in $(x_1,x_2)$, we have 
\[
\cD_{\cO_i}^s \cap \cDsr = \left\{ 0 \right\}, \quad 
H^s_{\cO_i} \cap \Hsr (\RN) = \left\{ 0 \right\}
\]
where $\Hsr := \set{u \in H^s(\RN) | u(x) = u(|x|) }$. 
In this setting, we have
	\begin{theorem}\label{theorem:1.5}
			\begin{enumerate}
				\item
				Under $N \geq 4$, \textup{(f0)}, \eqref{eq:1.3} and \textup{(f3)},
				\eqref{eq:1.1} has at least one 
				nonradial solution $u_0 \in \cD_{\cO_1}^s$. 
				\item 
				Assume \textup{(f0)--(f3)} and either $N =4$ or $N \geq 6$. Then 
				\eqref{eq:1.1} admits infinitely many nonradial solutions 
				$(u_k)_{k \geq 1} \subset \cD_{\cO_2}^s $ 
				with $|I_0(u_k)| < \infty$ and $\| u_k \|_{\cDs{\RN}} \to \infty$ as $k \to \infty$. 
				\item 
				Suppose $N \geq 4$, $a \geq 0$ and \textup{(F0)--(F3)}. 
				Then \eqref{eq:1.2} has at least one nonradial solution $u_0 \in H^s_{\cO_1}$. 
				\item 
				Let $N = 4$ or $N \geq 6$. Under $a \geq 0$ and \textup{(F0)--(F3)}, 
				there exist infinitely many 
				nonradial solutions $(u_k)_{k \geq 1} \subset H^s_{\cO_2} $ of \eqref{eq:1.2} 
				with $P_a(u_k) = 0$, $\| u_k \|_{H^s} \to \infty$ and $I_a(u_k) \to \infty$. 
				Here $P_a$ is defined by 
					\[
						\begin{aligned}
						P_a(u) &:= \frac{N-2s}{2} \| u \|_{s,a}^2 
						+ as \int_{\RN} (a+4\pi^2|\xi|^2)^{s-1} |\scrF u(\xi)|^2 \rd \xi - N \int_{\RN} F(u) \rd x.
						\end{aligned}
					\]
			\end{enumerate}
	\end{theorem}

	\begin{remark}
		The equality $P_a(u) = 0$ corresponds to the Pohozaev identity. 
	\end{remark}

	Here, we comment on the proofs of Theorems \ref{theorem:1.1}, \ref{theorem:1.3} and \ref{theorem:1.5}. 
We prove Theorems \ref{theorem:1.1}, \ref{theorem:1.3} and \ref{theorem:1.5} relying on abstract results, 
which are essentially proved in Hirata and Tanaka \cite{HT-18} and based on the properties of scaled functions 
introduced in \cite{J-97, HIT-10} for the $L^2$-constraint problem and the scalar field equations. 
One of the advantages of the abstract results is that 
we need not to construct an auxiliary functional used in \cite{HIT-10,Amb-18-1,I-17,JL-18}. 
This auxiliary functional is exploited to ensure that the obtained sequence of critical values diverge, and 
its construction depends on the problems and the function spaces. 
However, the abstract result found in \cite{HT-18} allows us to avoid this construction and 
we may treat different equations \eqref{eq:1.1} and \eqref{eq:1.2} in the more unified way. 
In this paper, the setting in \cite{HT-18 } is slightly extended and for this some modifications 
of the arguments are necessary. We believe that the result in this paper 
may be applied to other various operators, like the $p$-Laplacian. 
A similar abstract result covering $L^2$-constraint problems is developed in Ikoma and Tanaka \cite{IT-19}. 
We remark that in \cite{M-17}, Mederski also developed an abstract critical point theory and 
in \cite{BS-17,BS-18,BS-19}, Bartsch and Soave studied the $L^2$-constraint problem 
together with the minimax method in Ghoussoub \cite{G-93}. 
The approaches in \cite{BS-17,BS-18,BS-19,M-17} are close and related to our abstract result 
since they used the scaling of functions and its properties. 
See also Bartsch and de Valeriola \cite{BdS-13}.

	We add some comments for the proofs of Theorems \ref{theorem:1.1} and \ref{theorem:1.5} (ii), 
and differences between Ambrosio \cite{Amb-18-1} and our arguments. 
First, to the best of author's knowledge, it is not known whether or not 
the Pohozaev identity is satisfied by every solution of \eqref{eq:1.1}, in particular, 
for the case $0<s<1/2$ and $f \in C(\R)$. 
Therefore, we may not apply the arguments in \cite{BL-83-3,Stru-82,Stru-86} directly. 
Second, since $I_0$ is not well-defined on $\cDs{\RN}$ under (f0)--(f3), we need modifications of $f(t)$. 
To overcome these difficulties, Ambrosio \cite{Amb-18-1} required $f(t)$ to be of class $C^{1,\alpha}$. 
In addition, he approximated $f(t)$ 
by functions in the positive mass case, namely satisfying (F0)--(F3), as in Berestycki and Lions \cite{BL-83-3}. 
Therefore, the function space $H^s$ is used for approximated problems in \cite{Amb-18-1} and \cite{BL-83-3}. 
On the other hand, in this paper, we perform a different modification. 
Our modification allows us to work on $\cDs{\RN}$ and is easy to compare critical values. 
This point is useful to find infinitely many solutions. 
We also mention the recently announced paper \cite{M-19}. 
In \cite{M-19}, a similar modification of $f(t)$ to ours was used and Mederski worked on $\cD^1(\RN)$ 
and applied an abstract critical point theory in \cite{M-17}  to obtain infinitely many solutions.

	Let us point out that our arguments work for the case $s=1$. 
Therefore, the argument in this paper provides another proof 
for the results by \cite{BL-83-3,M-17,M-19,JL-18,Stru-82,Stru-86}.

	This paper is organized as follows. 
In section \ref{section:2}, we present and prove the abstract result inspired by \cite{HT-18}. 
Section \ref{section:3} is devoted to the proofs of Theorem \ref{theorem:1.1} and \ref{theorem:1.3}. 
Finally, in section \ref{section:4}, we prove Theorem \ref{theorem:1.5}.


\section{Abstract Results}\label{section:2}


	In this section, we state and prove an abstract result for proving Theorems 
\ref{theorem:1.1}, \ref{theorem:1.3} and \ref{theorem:1.5}. 
The result is motivated and essentially proved in \cite{HT-18}. 
Here we generalize the result slightly from that in \cite{HT-18}, but 
the argument is similar to \cite{HT-18}.

	Let $X$ be a Banach space and $I \in C^1(X,\R)$ be an even functional. 
Suppose that there is an action of $\R$ on $X$ and write $(\theta, u) \mapsto T_\theta u : \R\times X \to X$ 
for its action where $T_\theta \in \cL(X)$, 
$T_{\theta_1 + \theta_2} = T_{\theta_1} \circ T_{\theta_2}$ and 
$\cL(X)$ stands for the set consisting of bounded linear transformations on $X$. 
We further assume the following for $(T_\theta)_{\theta \in \R}$: 
	\begin{equation}\label{T1}
		\left\{\begin{aligned}
			&\text{A family $(T_\theta)_{\theta \in \R}$ is strongly continuous, that is,}\\
			&\text{for each $u \in X$, the map $\theta \mapsto T_\theta u : \R \to X$ is continuous}.
		\end{aligned}\right.
		\tag{T1}
	\end{equation}
Notice that by the uniform boundedness principle, \eqref{T1} and 
$(T_\theta)^{-1} = T_{-\theta}$, 
it may be verified that for every $m > 0$
	\begin{equation}\label{eq:2.1}
		\sup_{ |\theta | \leq m  } \left\| T_\theta \right\|_{\cL (X)} < \infty, 
		\quad 
		0 < \inf_{ |\theta | \leq m , \| u \|_X = 1 } \left\| T_\theta u \right\|_X.
	\end{equation}
In particular, the map $(\theta, u) \mapsto T_\theta u : \R \times X \to X $ is continuous.

	Next, put $J(\theta, u) := I( T_\theta u ) $. Since we do not assume the smoothness of the action, 
we only have $J \in C (\R \times X, \R)$, but as the second assumption, we suppose 
	\begin{equation}\label{T2}
		J(\theta, u) \in C^1( \R \times X, \R ).
		\tag{T2}
	\end{equation}
Typical examples satisfying \eqref{T1} and \eqref{T2} are given below. 
But we first notice that by 
$J(\theta + h, u) = I( T_{\theta +h} u ) = I (T_h \circ T_\theta u) = J(h,T_\theta u)$, 
it follows that 
	\[
		J(\theta + h, u) - J(\theta , u) = J( h, T_\theta u ) - J(0,T_\theta u) \quad 
		{\rm for\ each} \ \theta, h \in \R, \ u \in X.
	\]
Hence, we get 
	\begin{equation}\label{eq:2.2}
	\rd_\theta J( \theta_1 , u_1 ) = \rd_{\theta} J(0, T_{\theta_1} u_1 ) \quad 
		{\rm for\ every \ } (\theta_1, u_1) \in \R \times X.
	\end{equation}
Remark also that $\rd J(0,u) = 0$ implies $\rd I(u) = 0$. 
Therefore, 
in what follows, we look for critical points $(0,u)$ of $J$ 
instead of those of $I$. 

	Now we give examples which enjoy \eqref{T1} and \eqref{T2}. 

(i) Let $\cO \subset O(N)$ be a closed subgroup and 
	\[
		X := \Set{ u \in H^s(\RN) | u(g x) = u(x) \ {\rm for\ all} \ g \in \cO }
	\]
with a norm induced from $H^s(\RN)$. 
Set $ T_\theta u (x) := u( e^{-\theta} x )$ and 
	\[
		I(u) := \frac{1}{2} \| u \|_{\cDs{\RN}}^2 - \int_{\RN} F(u) \rd x
	\]
where $f$ satisfies (F0)--(F3) with $ - \infty < \liminf_{|t| \to \infty} |t|^{2-2^\ast_s} t^{-1} f(t)$. 
In this case, one easily sees that the map $\theta \mapsto T_\theta u : \R \to X$ 
is continuous for each $u \in X$ and 
	\[
		J(\theta, u) = I ( T_\theta u ) 
		= \frac{e^{(N-2s) \theta  }}{2} \| u \|_{\cDs{\RN}}^2 - e^{N\theta} \int_{\RN} F(u) \rd x 
		\in C^1( \R \times X , \R ).  
	\]
Thus, \eqref{T1} and \eqref{T2} hold. 

(ii) In (i), replace $H^s(\RN)$ by $\cD^s(\RN)$, namely, consider 
	\[
		X := \Set{ u \in \cD^s (\RN) | u(gx) = u(x) \ {\rm for \ all}\ g \in \cO }.
	\]
Let $T_\theta$, $I$ and $J$ be as in the above 
where $f$ satisfies (f0)--(f3) with $ - \infty < \liminf_{|t| \to 0} |t|^{2-2^\ast_s} t^{-1} f(t)$ and 
$- \infty < \liminf_{|t| \to \infty} |t|^{2-2^\ast_s} t^{-1} f(t)$. 
Then it may be checked that \eqref{T1} and \eqref{T2} hold.

	\begin{remark}\label{remark:2.1}
		In the above examples, the equality $0 = \rd_\theta J(0,u)$ is equivalent 
		to the Pohozaev identity. 
	\end{remark}

As in \cite{HT-18}, we introduce the following compact condition:
	\begin{definition}\label{definition:2.2}
		Let $c \in \R$. The functional $I$ is said to satisfy \PSPc\  if
		for every $(u_n)_{n=1}^\infty \subset X$ with 
		$J(0,u_n) = I(u_n) \to c$ and $\| \rd J(0,u_n) \|_{\R \times X^\ast} \to 0$, 
		there exists a strongly convergent subsequence $(u_{n_k})_{k=1}^\infty$ in $X$. 
	\end{definition}

Then our abstract results are the following

	\begin{theorem}\label{theorem:2.3}
		Let $X$ be a Banach space, $I \in C^1(X,\R)$, and \eqref{T1} and \eqref{T2} be satisfied. 
		Assume also that there exist $\rho_0>0$ and $w_0 \in X$ such that 
			\[
				I(0) = 0   < \inf_{ \| u \|_X = \rho_0} I (u), \quad 
				I(w_0) < 0.
			\]
		Then there exists $(u_n)_{n=1}^\infty \subset X$ such that 
		$I(u_n) \to c_{mp}$ and $\| \rd J(0,u_n) \|_{\R \times X^\ast} \to 0$ 
		where 
			\[
				c_{mp} := \inf_{ \gamma \in \Gamma} \max_{0 \leq t \leq 1} I( \gamma (t) ), \quad 
				\Gamma := \Set{ \gamma \in C( [0,1], X ) | \gamma(0) = 0, \ \rho_0< \| \gamma (1) \|_X, \   I(\gamma(1) ) < 0  }.
			\]
		In addition, if $\mathrm{(PSP)}_{c_{mp}}$ holds, then 
		there exists a $u_0 \in X$ such that $\rd J(0,u_0) = 0$ and $I(u_0) = J(0,u_0) = c_{mp}$. 
	\end{theorem}

Next, we state a symmetric mountain pass version:

	\begin{theorem}\label{theorem:2.4}
		Let $X$ be a Banach space and $I \in C^1(X,\R)$ with $I(-u) = I(u)$ for any $u \in X$. 
		Suppose \eqref{T1} and \eqref{T2}, and that 
		$I$ satisfies \PSPc\   for each $c > 0$ 
		as well as the symmetric mountain pass geometry: 
			\begin{enumerate}
				\item
					There exists a $\rho_0>0$ such that 
					$ \inf_{\| u \|_{X} = \rho_0} I(u) > 0 = I(0)$. 
				\item
					For each $j \geq 1$, there exists an odd map 
					$\gamma_j \in C( S^{j-1} , X )$ such that 
					$\max_{ \sigma \in S^{j-1} } I(\gamma_j(\sigma)) < 0$ and 
					$\gamma_j (S^{j-1}) \subset X \setminus \ov{B_{\rho_0} (0) }$ 
					where $S^{j-1} := \set{ \sigma \in \R^j |   |\sigma|=1 }$. 
			\end{enumerate}
		Then there exist $(u_j)_{j=1}^\infty$ so that 
		$\rd J(0,u_j) = 0$ and $I(u_j) = J(0,u_j) = c_j \to \infty$ as $j \to \infty$ 
		where $c_j$ is defined in \eqref{eq:2.27} below. 
	\end{theorem}

Before proving Theorems \ref{theorem:2.3} and \ref{theorem:2.4}, 
we first prepare a deformation lemma introduced in \cite{HT-18}.


\subsection{Deformation Lemma}
\label{section:2.1}


Let $X$ be a Banach space, and 
$(T_\theta)_{\theta \in \R}$ and $I \in C^1(X,\R)$ satisfy 
\eqref{T1} and \eqref{T2}. 
The aim of this subsection is to prove

	\begin{lemma}\label{lemma:2.5}
		For $c > 0$, define 
		\[
			K_c := \Set{ u \in X | I(u) = J(0,u) = c, \ \rd J(0,u) =0 }. 
		\]
		Suppose that $I$ satisfies \PSPc and 
		let $\cO \subset X $ be any neighborhood of $K_c$ and $\cO = \emptyset$ provided $K_c = \emptyset$. 
		Then for any $ \be > 0$, there exist $\e \in (0, \be )$ and 
		$\eta \in C( [0,1] \times X , X )$ such that 
		\begin{enumerate}
			\item[\rm (i)] $\eta (0,u) = u$ for every $u \in X$. 
			\item[\rm (ii)] $\eta (t,-u) = -\eta (t,u)$ for each $(t,u) \in [0,1] \times X$ if $I$ is even. 
			\item[\rm (iii)] $\eta (t,u) = u$ if $u \in [I \leq c - \be ] := \set{ u \in X | I(u) \leq c - \be }$. 
			\item[\rm (iv)] A map $t \mapsto I(\eta (t,u)) : [0,1] \to \R$ is nonincreasing for any $u \in X$. 
			\item[\rm (v)] $\eta (1 , [ I \leq b + \e ] \setminus \cO ) \subset [ I \leq b - \e ]$ and 
			$\eta (1, [ I \leq b + \e] ) \subset [  I \leq b - \e ] \cup \cO$. 
		\end{enumerate}
	\end{lemma}

To prove Lemma \ref{lemma:2.5}, as in \cite{HT-18}, 
we exploit the functional $J(\theta, u)$ and first 
prove a deformation lemma for $J$. 
For this purpose, we need some preparations.

	Let us write
	\[
		\| (\theta, u) \|_{\R \times X} := | \theta | + \| u \|_X, \quad 
		d_{\R \times X} \left( (\theta_1,u_1) , (\theta_2, u_2 ) \right) := 
		\left\| (\theta_1-\theta_2 , u_1 - u_2 ) \right\|_{\R \times X}.
	\]
Next, we define the second metric $d_0$ by
	\[
		d_0 \left(  \left( \theta_1,u_1  \right) ,  \left( \theta_2,u_2  \right)\right) 
		:= \inf \left\{  \int_0^1 \left\| \dot{\sigma} (\tau) \right\|_{ \sigma(\tau) } \rd \tau 
		\;\middle| \; \sigma \in C^1( [0,1], \R \times X ), \ \sigma (0) = \left( \theta_1, u_1 \right), \ 
		\sigma(1) = \left( \theta_2, u_2 \right)  \right\}
	\]
where 
	\[
		\left\| \left( \wt{\theta} , \wt{u} \right) \right\|_{(\theta,u)} 
		:= \left( \wt{\theta}^2 + \left\| T_{\theta} \wt{u} \right\|_X^2 \right)^{1/2} 
		\quad {\rm for} \ (\theta,u), (\wt{\theta}, \wt{u}) \in \R \times X .
	\]
Since $(\theta, u) \mapsto \| T_\theta u \|_X$ is continuous, 
so is $\tau \mapsto \| \dot{\sigma} (\tau) \|_{\sigma (\tau)}$ and 
$d_0$ is well-defined. Furthermore, we have 
	\begin{proposition}\label{proposition:2.6}
		\begin{enumerate}
			\item 
				The metric space $(\R \times X,d_0)$ is complete. 
			\item 
				For every $m > 0$ and $R>0$, there exist $C_{m,R}, \wt{C}_{m,R} > 0$ such that 
				for each $(\theta,u) \in [-m,m] \times X$ and $r \in (0,R]$, 
					\begin{align}
						& d_{\R \times X} \left(  (\theta,u) , (\theta_1,u_1) \right) 
						\leq C_{m,R} d_0 \left( (\theta,u) , (\theta_1,u_1) \right)
						\label{align:2.3}
						& &\textup{for\ all\ } (\theta_1,u_1) \in B^{d_0}_r(\theta,u),
						\\
						& d_{0} \left(  (\theta,u) , (\theta_1,u_1) \right) 
						\leq \wt{C}_{m,R} d_{\R \times X} \left( (\theta,u) , (\theta_1,u_1) \right)
						& &\textup{for\ all\ } (\theta_1,u_1) \in B^{d_{\R\times X}}_{ \wt{C}_{m,R}^{-1}\, r }(\theta,u)
						\label{align:2.4}
					\end{align}
			where $B^d_r(\theta,u)$ denotes a ball centered at $(\theta,u)$ with radius $r>0$ 
			in a metric $d$. 
			In particular, 
				\[
					B_{ \wt{C}_{ m ,R }^{-1} \, r  }^{d_{\R \times X}} (\theta, u) 
					\subset B_{r}^{d_0}(\theta,u) 
					\subset B_{ C_{ m, R } \, r  }^{d_{\R \times X}} (\theta, u) 
					\quad {\rm for\ each} \ r \in (0,R].
				\]
			\item[\emph{(iii)}] For all $\alpha \in \R$ and $(\theta_1,u_1), (\theta_2,u_2) \in \R \times X$, 
				\begin{equation}\label{eq:2.5}
					d_0 \left( \left( \theta_1, u_1 \right) , \left( \theta_2 , u_2 \right) \right) 
					= d_0 \left( \left( \theta_1 + \alpha , T_{-\alpha} u_1 \right) , 
					\left( \theta_2 + \alpha,  T_{-\alpha} u_2 \right) \right). 
				\end{equation}
			\item[\emph{(iv)}] 
			For each $u,v \in X$, 
				\begin{equation}\label{eq:2.6}
					d_0 \left( \left( 0,u \right) , \left( 0 ,v  \right) \right) 
					\leq  \left\| v - u \right\|_X.
				\end{equation}
		\end{enumerate}
	\end{proposition}

	\begin{proof}
(i) Remark that for any $\sigma (\tau) = ( \sigma_\theta (\tau), \sigma_u (\tau ) ) \in C^1([0,1], \R \times X)$ 
and $s \in [0,1]$, we have 
	\begin{equation}\label{eq:2.7}
		| \sigma_\theta (s) - \sigma_\theta (0) | \leq \int_0^s | \dot{\sigma}_\theta (\tau) | \rd \tau 
		\leq \int_0^1 \| \dot{\sigma} (\tau) \|_{\sigma (\tau)} \rd \tau. 
	\end{equation}

	Let $( (\theta_n, u_n)  )_n$ be a Cauchy sequence in $(\R \times X, d_0)$ and 
$\sigma_{n,m} \in C^1( [0,1], \R \times X )$ satisfy 
	\begin{equation}\label{eq:2.8}
		\int_0^1 \| \dot{\sigma}_{n,m} (\tau)  \|_{\sigma_{n,m} (\tau) } \rd \tau 
		\leq 2 d_0 \left( (\theta_n,u_n) , (\theta_m,u_m) \right), \quad 
		\sigma_{n,m} (0) = (\theta_n,u_n), \quad 
		\sigma_{n,m} (1) = ( \theta_m, u_m ). 
	\end{equation}
From \eqref{eq:2.7} and \eqref{eq:2.8}, it follows that $(\theta_n)_n$ is a Cauchy sequence in $\R$. 
Hence, there exists a $\theta_\infty \in \R$ such that 
	\[
		\max_{0 \leq s \leq 1} \left| \sigma_{n,m,\theta} (s) - \theta_\infty \right| 
		\to 0 \quad {\rm as} \ n,m \to \infty. 
	\]
Thus, $(\max_{0 \leq s \leq 1} |\sigma_{n,m,\theta} (s) |)_{n,m}$ is bounded and 
$\theta_n \to \theta_{\infty}$. 
By \eqref{eq:2.1}, for some $c>0$, we obtain 
	\begin{equation}\label{eq:2.9}
		\begin{aligned}
			2 d_0 \left( (\theta_n,u_n) , (\theta_m,u_m) \right) 
			& \geq 
			\int_0^1 \| \dot{\sigma}_{n,m} (\tau) \|_{\sigma_{n,m} (\tau)} \rd \tau 
			\\
			& \geq \int_0^1 \| T_{\sigma_{n,m,\theta} (\tau)  } \dot{\sigma}_{n,m,u} (\tau) \|_X \rd \tau 
			\\
			& \geq c \int_0^1 \| \dot{\sigma}_{n,m,u} (\tau) \|_X \rd \tau 
			\\
			& \geq c \left\| \int_0^1 \dot{\sigma}_{n,m,u} (\tau) \rd \tau \right\|_X 
			= c \| u_n - u_m \|_X.
		\end{aligned}
	\end{equation}
This implies that $(u_n)_n$ is a Cauchy sequence in $(X,\| \cdot \|_X)$ and 
$\| u_n - u_\infty \|_X \to 0$ for some $u_\infty \in X$.

	Now by the straight line joining $(\theta_n,u_n)$ and $(\theta_\infty, u_\infty)$ 
and \eqref{eq:2.1}, we may find a $C>0$ such that 
	\[
		d_0 \left( (\theta_n,u_n) , (\theta_\infty , u_\infty) \right) 
		\leq C \left( | \theta_n - \theta_\infty | + \| u_n - u_\infty \|_X \right) \to 0.
	\]
Hence, $(\R \times X , d_0)$ is complete.

	(ii) Let $(\theta,u) \in [-m,m] \times X$, $r \in (0,R]$ and 
	$(\theta_1,u_1) \in B^{d_0}_r ( \theta, u )$. 
Let $\sigma \in C^1([0,1], \R \times X)$ satisfy 
	\[
		\sigma(0) = (\theta,u), \quad \sigma (1) = (\theta_1, u_1), \quad 
		\int_0^1 \| \dot{\sigma} (\tau) \|_{\sigma (\tau)} \rd \tau 
		\leq 2 d_0 \left( (\theta, u) , ( \theta_1, u_1) \right) <  2R. 
	\]
By \eqref{eq:2.7}, we obtain 
	\[
		| \theta_1 - \theta | \leq 2 d_0 \left( (\theta, u) , ( \theta_1, u_1) \right), 
		\quad 
		| \sigma_\theta (s) | \leq | \theta | + 2R \leq m + 2 R \quad 
		{\rm for\ each \ } s \in [0,1].
	\]
Notice that from a similar argument to \eqref{eq:2.9}, we may find 
a $c_{m,R}>0$ so that 
	\[
		\| u_1 - u \|_X \leq c_{m,R}^{-1} \int_0^1 \| \dot{\sigma} (\tau) \|_{\sigma(\tau)} \rd \tau 
		\leq 2 c_{m,R}^{-1} \, d_0 \left( (\theta, u) , ( \theta_1, u_1) \right) .
	\]
Hence, 
	\[
		\begin{aligned}
			d_{\R \times X}  \left( (\theta_1, u_1) , (\theta,u)  \right) 
			=  | \theta - \theta_1| + \| u - u_1 \|_X 
			\leq 2 \left( 1 + c_{m,R}^{-1} \right) d_0 \left( (\theta, u) , ( \theta_1, u_1) \right) 
		\end{aligned}
	\]
and \eqref{align:2.3} holds.

	On the other hand, 
let $(\theta,u) \in [-m,m] \times X$, $r \in (0,R]$ and select $\wt{C}_{m,R} \geq 1$ so that 
$\sup_{ |\theta | \leq m + R  } \| T_\theta \|_{ \cL (X) } \leq \wt{C}_{m,R}$. 
Remark that 
	\[
		| \theta + \tau (\theta_1 - \theta) | \leq m + \tau \wt{C}_{m,R}^{-1} r \leq m+ R \quad 
		\text{for any $\tau \in [0,1]$ and $(\theta_1,u_1) \in B^{d_{\R \times X}}_{ \wt{C}^{-1}_{m,R} \, r } (\theta, u)$}.
	\]
Joining $(\theta,u)$ and $(\theta_1,u_1)$ 
by the straight line, we find that 
	\[
		\begin{aligned}
			d_0 ( (\theta, u) , ( \theta_1, u_1 ) ) 
			& \leq \int_0^1 \left( |\theta_1-\theta|^2 + 
			\| T_{ \tau \theta_1 + (1-\tau) \theta } (u_1-u) \|_X^2 \right)^{1/2}
			\rd \tau 
			\\
			& \leq \wt{C}_{m,R} \int_0^1 
			\left( |\theta_1 - \theta| + \| u_1 - u \|_X \right) \rd \tau 
			= \wt{C}_{m,R} \, d_{\R \times X} 
			\left( (\theta_1, u_1) , (\theta, u) \right).
		\end{aligned}
	\]
Hence, \eqref{align:2.4} holds.

	(iii) \eqref{eq:2.5} follows from the fact 
	\[
		\left\| \left( \wt{\theta} , \wt{u} \right) \right\|_{(\theta,u)}^2 
		= \wt{\theta}^2 + \left\| T_{\theta + \alpha}  T_{-\alpha} \wt{u}  \right\|^2_{X} 
		=  \left\| \left( \wt{\theta} , T_{-\alpha} \wt{u} \right) \right\|_{(\theta+\alpha,u)}^2 
		\quad {\rm for\ every\ } \alpha \in \R, \ \left(\theta,u\right), \left(\wt{\theta},\wt{u} \right) \in \R \times X.
	\]

	(iv) Joining $(0,u)$ and $(0,v)$ by the straight line, 
we see that \eqref{eq:2.6} holds. 
	\end{proof}

	\begin{remark}\label{remark:2.7}
		As a consequence of Proposition \ref{proposition:2.6} (ii), 
		both of the continuity and 
		the local Lipschitz continuity for maps 
		are equivalent in $d_0$ and $d_{\R \times X}$.
	\end{remark}

	Next, for $F \in (\R \times X)^\ast$ and $(\theta,u) \in \R \times X$,  define $\| \cdot \|_{(\theta,u), \ast}$ by 
	\[
		\left\| F \right\|_{ (\theta, u) , \ast } := \sup \Set{ \la F , \left( \wt{\theta} , \wt{u} \right) \ra 
		 |  \left\| \left( \wt{\theta} , \wt{u} \right) \right\|_{ (\theta, u) } \leq 1}.
	\]
In what follows, we shall use $\| \cdot \|_{(\theta,u),\ast}$ for $\rd J(\theta,u)$. 
By noting $J(\theta, u) = I(T_\theta u) = J(0,T_\theta u)$ and 
	\[
		\rd_u J(\theta, u) [ \wt{u} ] = \rd I ( T_\theta u ) \left[ T_{\theta} \wt{u} \right] 
		= \rd_u J(0, T_\theta u) \left[ T_\theta \wt{u} \right], 
	\]
it follows from \eqref{eq:2.2} that 
	\begin{equation}\label{eq:2.10}
		\rd J(\theta,u) \left[ \left( \wt{\theta} , \wt{u} \right) \right] 
		= \rd_\theta J( \theta, u) \wt{\theta} 
		 + \rd I (T_\theta u) \left[ T_\theta \wt{u} \right] 
		 = \rd_\theta J( 0 , T_\theta u ) \wt{\theta} 
		 + \rd_u J (0, T_\theta u) \left[ T_\theta \wt{u} \right],
	\end{equation}
which implies 
	\begin{equation}\label{eq:2.11}
		\begin{aligned}
			&\left\| \rd J(\theta, u) \right\|_{ (\theta,u) , \ast } 
			= \sqrt{
				\left| \rd_{\theta} J(0,T_\theta u) \right|^2 + 
				\left\| \rd J (0,T_\theta u) \right\|_{X^\ast}^2 
			},
			\\
			&\text{the map $(\theta,u) \mapsto \left\| \rd J(\theta,u) \right\|_{(\theta,u),\ast} : 
			\R \times X \to \R $ is continuous}
		\end{aligned}
	\end{equation}
Furthermore, it is not difficult to check that for $c \in \R$ and $\delta > 0$,
	\begin{equation}\label{eq:2.12}
		0 < \inf \left\{
		\left\| \rd J (0, u) \right\|_{(\R \times X)^\ast} 
		\ \big|\ \left| J(0,u) - c \right| < \delta 
		\right\}
		\ \Leftrightarrow \ 
		0 < \inf \left\{
		\left\| \rd J (\theta, u) \right\|_{(\theta,u),\ast} 
		\ \big|\ \theta \in \R, \ \left| J(\theta,u) - c \right| < \delta 
		\right\}.
	\end{equation}

For $c,\delta,\rho > 0$, set 
	\[
		\begin{aligned}
			\wt{K}_c &:= \Set{ (\theta, u) \in \R \times X | J(\theta, u) = c, \ \rd J(\theta, u) = 0 },
			\\
			\wt{K}_{c,\delta,\rho} &:= 
			\Set{ (\theta, u) \in \R \times X |  
				\;
			| J(\theta, u) - c | < \delta, \ d_0 \left( (\theta, u) , \wt{K}_c \right) \geq \rho
			 }. 
		\end{aligned}
	\]
Notice that  \eqref{eq:2.10} shows that for every $\theta,\theta_1 \in \R$, 
	\begin{equation}\label{eq:2.13}
		u \in K_c \quad \Leftrightarrow \quad (0, u ) \in \wt{K}_c, \qquad 
		\quad 
		(\theta,u) \in \wt{K}_c \quad 
		\Leftrightarrow \quad (0,T_\theta u) \in \wt{K}_c \quad \Leftrightarrow \quad 
		(\theta + \theta_1, T_{-\theta_1} u) \in \wt{K}_c.
	\end{equation}

	\begin{proposition}\label{proposition:2.8}
		Assume that \PSPc \  holds. 
		\begin{enumerate}
			\item[\rm (i)] 
				If $( (\theta_n,u_n) )_n \subset \R \times X$ is a Palais--Smale sequence of $J$ at level $c$, that is, 
				$J(\theta_n,u_n) \to c$ and $\| \rd J (\theta_n,u_n) \|_{(\theta_n,u_n),\ast} \to 0$,  
				then $(T_{\theta_n} u_n)_n$ has a strongly convergent subsequence in $X$, $K_c \neq \emptyset$ and 
				$d_0( (\theta_n,u_n) , \wt{K}_c ) \to 0$.  
				In particular, if $K_c = \emptyset$ (equivalently $\wt{K}_c = \emptyset$), then there exists a $\delta_0>0$ 
				such that 
					\begin{equation}\label{eq:2.14}
						\delta_0 \leq \inf \left\{  \| \rd J(\theta, u) \|_{(\theta,u),\ast}  \;\middle|\;
						|J(\theta,u) - c | < \delta_0 \right\}. 
					\end{equation}
			\item[\rm (ii)] 
				Suppose $K_c \neq \emptyset$ (equivalently $\wt{K}_c \neq \emptyset$). 
				For each $\rho >0$, there exists a $\delta_\rho > 0$ such that 
				either $\wt{K}_{c,\delta_\rho,\rho/3} = \emptyset$ 
				or else $\wt{K}_{c,\delta_\rho , \rho / 3} \neq \emptyset$ and 
					\begin{equation}\label{eq:2.15}
						\delta_\rho \leq \inf \left\{ \| \rd J(\theta, u) \|_{(\theta,u),\ast} 
						\;\middle|\;
						(\theta, u) \in \wt{K}_{c,\delta_\rho,\rho/3} \right\}.
					\end{equation}
		\end{enumerate}
	\end{proposition}

	\begin{proof}
(i) Write $v_n := T_{\theta_n} u_n$ and notice that $I(v_n) = J(\theta_n,u_n) \to c$, 
$ \rd_\theta J(0,v_n) = \rd_\theta J(\theta_n,u_n)  \to 0$ and 
$\| \rd J(0,v_n) \|_{X^\ast} \to 0$ thanks to \eqref{eq:2.2} and \eqref{eq:2.11}. 
Thus, \PSPc asserts that 
$(v_n)_n$ has a strongly convergent subsequence in $X$ and its limit belongs to $K_c$. 
If there exist $\e_0>0$ and a subsequence $(\theta_{n_k}, u_{n_k})$ such that 
$d_0 (( \theta_{n_k}, u_{n_k} ), \wt{K}_c) \geq \e_0$, as in the above, subtracting a subsequence 
further if necessary (we still write $(n_k)$), we find that $T_{\theta_{n_k}} u_{n_k} \to v_0$ strongly in $X$ 
where $v_0 \in K_c$. Now from \eqref{eq:2.5}, \eqref{eq:2.6}, $v_0 \in K_c$ and \eqref{eq:2.13}, it follows that 
	\[
		\e_0 \leq d_0 \left( (\theta_{n_k}, u_{n_k}) , \wt{K}_c \right) 
		= d_0 \left( (0, T_{\theta_{n_k}} u_{n_k}  ) , \wt{K}_c \right) 
		\leq \| T_{\theta_{n_k}} u_{n_k} - v_0 \|_X \to 0,
	\]
which is a contradiction. Thus, $d_0( (\theta_n,u_n) , \wt{K}_c ) \to 0$.

	Finally, if $K_c = \emptyset$, then the above fact implies that 
there is no Palais--Smale sequence of $J$ at level $c$. Therefore, \eqref{eq:2.14} holds.

(ii) Fix $\rho >0$. We first remark 
$ \wt{K}_{c,\delta_1 , \rho/3} \subset \wt{K}_{c,\delta_2,\rho/3}$ for any $\delta_1 < \delta_2$. 
Hence, either there exists a $\wt{\delta}_\rho >0$ such that $\wt{K}_{c,\delta,\rho/3} = \emptyset$ 
for every $\delta \in (0,\wt{\delta}_\rho)$, or else $\wt{K}_{c,\delta,\rho/3} \neq \emptyset$ 
for each $\delta >0$. 
Thus, it suffices to show \eqref{eq:2.15} when $\wt{K}_{c,\delta,\rho/3} \neq \emptyset$ 
for any $\delta > 0$. 
We argue indirectly and suppose that there exists $(\theta_n,u_n) \in \wt{K}_{c,n^{-1},\rho/3}$ such that 
$\| \rd J(\theta_n,u_n) \|_{(\theta_n,u_n),\ast} \leq n^{-1} \to 0$. 
By (i) and $(\theta_n,u_n) \in \wt{K}_{c,n^{-1},\rho/3}$, 
we obtain a contradiction: $0 < \rho/3 \leq d_0( (\theta_n,u_n) , \wt{K}_c ) \to 0$. Hence, \eqref{eq:2.15} holds. 
	\end{proof}

	Now we state a deformation lemma for $J(\theta,u)$ under \PSPc.

	\begin{lemma}\label{lemma:2.9}
		For $\rho > 0$, put
			\[
				\wt{N}_\rho ( \wt{K}_c ) := \Set{ (\theta, u) \in \R \times X  |
				d_0 \left( (\theta, u) , \wt{K}_c \right) < \rho }
			\]
		and suppose that \PSPc\  holds. 
		Then for each $\be, \rho > 0$ there exist $\e \in (0,\be)$ 
		and $\wt{\eta} \in C( [0,1] \times \R \times X , \R \times X )$ such that 
		\begin{enumerate}
			\item[\rm (i)] 
			$\wt{\eta} (0,\theta,u) = (\theta,u)$ for each $(\theta,u) \in \R \times X$. 
			\item[\rm (ii)] 
			If $I$ is even, then 
			for each $(t,\theta,u) \in \R \times \R \times X$, 
			$\wt{\eta}_\theta (t,\theta,-u) = \wt{\eta}_\theta (t,\theta,u)$ and 
			$\wt{\eta}_u (t,\theta,-u) = - \wt{\eta}_u (t,\theta,u)$ where 
			$\wt{\eta} (t,\theta,u) = (\wt{\eta}_\theta (t,\theta,u) , \wt{\eta}_u (t,\theta , u) )$. 
			\item[\rm (iii)] 
			$\wt{\eta} (t,\theta,u) = (\theta, u)$ provided 
			$(\theta,u) \in [J \leq c - \be] := \set{ (\theta,u) \in \R \times X | J(\theta,u) \leq c - \be}$. 
			\item[\rm (iv)] 
			A map $t \mapsto J(\wt{\eta}(t,\theta,u)) : [0,1] \to \R$ is nonincreasing for any $(\theta,u) \in \R \times X$. 
			\item[\rm (v)] 
			$\wt{\eta} (1, [J \leq c + \e] \setminus \wt{N}_\rho ( \wt{K}_c ) ) \subset [ J \leq c - \e]$ and 
			$\wt{\eta} (1, [J \leq c + \e]) \subset [ J \leq c - \e ] \cup \wt{N}_\rho ( \wt{K}_c )$. 
			When $\wt{K}_c = \emptyset$, $\wt{\eta} (1, [ J \leq c + \e]) \subset [J\leq c - \e]$ holds.  
		\end{enumerate}
	\end{lemma}

	\begin{proof}
Put $ M := \set{ (\theta, u) \in \R \times X | \rd J (\theta, u) \neq 0 }$. 
Following Palais \cite{P-66} and noting Proposition \ref{proposition:2.6}, Remark \ref{remark:2.7} and 
the continuity of $(\theta,u) \mapsto T_\theta u$, 
we may find a locally Lipschitz continuous vector field 
$W$ on $ M $ such that for each $(\theta,u) \in M $, 
	\begin{equation}\label{eq:2.16}
		\left\| W (\theta, u) \right\|_{(\theta,u)} \leq \frac{3}{2} \left\| \rd J(\theta, u) \right\|_{(\theta,u), \ast} , 
		\quad 
		\rd J(\theta,u) W (\theta, u) \geq \left\| \rd J(\theta,u) \right\|_{(\theta,u),\ast}^2
		.
	\end{equation}
If $I$ is even, then by $J(\theta, - u) = I (T_\theta (-u)) = I(-T_\theta u) = J(\theta, u)$, 
we may also assume that $W$ satisfies 
	\[
		W_\theta (\theta, -u ) = W_\theta (\theta, u), \quad 
		W_u (\theta, - u) = - W_u (\theta, u)
	\]
where $W (\theta,u) = ( W_\theta (\theta,u) , W_u(\theta,u) )$. 
In fact, replace $W$ by the following:
	\[
		\left( \frac{1}{2} \left[ W_\theta (\theta, u) + W_\theta (\theta, -u) \right], 
		\ 
		\frac{1}{2} \left[ W_u(\theta, u) - W_u(\theta, - u) \right]
		 \right). 
	\]
In addition, using a similar argument to the above, we may also find 
a locally Lipschitz function $w : M \to (0,\infty)$ such that 
for every $(\theta, u) \in M$, 
	\begin{equation}\label{eq:2.17}
		\left| w(\theta, u) - \left\| W(\theta, u) \right\|_{(\theta,u)} \right| 
		\leq \frac{ \left\| W(\theta, u) \right\|_{(\theta,u)}  }{3}, 
	\end{equation}
and if $I$ is even, then we also have $w(\theta,-u) = w(\theta, u)$.

	Firstly, we treat the case $\wt{K}_c \neq \emptyset$. 
Let $\delta_\rho > 0$ be a number appearing in Proposition \ref{proposition:2.8} (ii). 
Remark that $K_{c,\delta_\rho, \rho/3} \subset M $ if $K_{c,\delta_\rho,\rho/3} \neq \emptyset$ 
and that by shrinking $\be >0$ if necessary, we may assume 
	\begin{equation}\label{eq:2.18}
		\be < \delta_\rho. 
	\end{equation}

	Next, we may choose a locally Lipschitz continuous function $\varphi : \R \times X \to [0,1]$ so that 
	\begin{equation}\label{eq:2.19}
		\varphi \equiv 1 \quad {\rm on} \ M \setminus \wt{N}_{\frac{2\rho}{3}} ( \wt{K}_c ), 
		\quad \varphi \equiv 0 \quad {\rm on} \ \wt{N}_{\frac{\rho}{3}} ( \wt{K}_c ), 
		\quad \text{and} \quad 
		\varphi ( \theta, - u) = \varphi (\theta, u) \ \text{if $I$ is even}.
	\end{equation}
Indeed, since $(\theta, u) \in \wt{K}_c$ (resp. $(\theta, u) \in \wt{N}_\rho (\wt{K}_c)$) is 
equivalent to $(\theta, -u) \in \wt{K}_c$ (resp. $(\theta, -u) \in \wt{N}_\rho (\wt{K}_c)$) 
provided $I$ is even, we set 
	\[
		\begin{aligned}
			\psi_1(\theta, u)& := \inf \left\{ \left\| (\theta ,u) - (\wt{\theta} , \wt{u}) \right\|_{\R \times X} 
			\;\middle|\; (\wt{\theta} , \wt{u}) \not \in \wt{N}_{\frac{5\rho}{9}} ( \wt{K}_c )  \right\},
			\\
			\psi_2(\theta, u)& := \inf \left\{ \left\| (\theta ,u) - (\wt{\theta} , \wt{u}) \right\|_{\R \times X} 
			\;\middle|\; (\wt{\theta} , \wt{u})  \in \wt{N}_{\frac{4\rho}{9}} ( \wt{K}_c )  \right\}
			.
		\end{aligned}
	\]
Then it is easily seen that $\psi_i$ are 
locally Lipschitz on $\R \times X$ and $\psi_i(\theta, -u) = \psi_i(\theta,u)$ 
if $I$ is even. Thus, 
	\[
		\varphi(\theta, u) := \frac{\psi_2(\theta, u)}{\psi_1(\theta, u) + \psi_2 (\theta, u)}
	\]
satisfies \eqref{eq:2.19}.

	Next, we pick up a $\psi \in C^1(\R, \R)$ so that $ 0 \leq \psi(t) \leq 1$ on $\R$ and 
	\begin{equation}\label{eq:2.20}
		\psi (t) = 1 \quad {\rm if} \ t \in \left[ c - \frac{\be}{2} , c + \frac{\be}{2} \right], \quad 
		\psi (t) = 0 \quad {\rm if} \ t \in (-\infty , c - \be ] \cup [c+\be,\infty)
		.
	\end{equation}
We consider the following ODE:
	\begin{equation}\label{eq:2.21}
		\frac{\rd \wt{\eta}}{\rd t} 
		= - \varphi ( \wt{\eta} (t) ) \psi \left( J( \wt{\eta} (t) ) \right) \frac{W(\wt{\eta}(t))}
		{w ( \wt{\eta} (t) )  }, \quad 
		\wt{\eta} (0,\theta,u) = (\theta, u). 
	\end{equation}
By Proposition \ref{proposition:2.8} (ii), \eqref{eq:2.17}, \eqref{eq:2.18} and \eqref{eq:2.20}, 
we see that the right hand side of \eqref{eq:2.21} is well-defined, bounded 
and locally Lipschitz on $\R \times X$, 
hence, \eqref{eq:2.21} is uniquely solvable for every $(\theta, u ) \in \R \times X$ and 
$\wt{\eta} \in C(\R \times \R \times X, \R \times X)$.

	By the choices of $\varphi, \psi, w$ and $W$,  properties (i)--(iv) are easily checked. 
Therefore, we will only prove (v) for some $\e \in (0,\be)$. We first fix $\e $ so that 
	\begin{equation}\label{eq:2.22}
		0 < \e < \min \left\{ \frac{\be}{2} , \frac{\delta_\rho}{4} , \frac{\rho \delta_\rho}{18} \right\}.
	\end{equation}
Let $(\theta, u) \in [J \leq c + \e] $. Our next task is to show 
$\wt{\eta} (1,\theta, u) \in [J \leq c - \e] \cup \wt{N}_{\rho}( \wt{K}_c)$. 
If $\eta (1,\theta,u) \not \in [J \leq c - \e]$, then 
it follows from (iv) that
	\begin{equation}\label{eq:2.23}
		c - \e \leq J( \wt{\eta} (t , \theta,u) ) \leq c + \e \quad 
		{\rm for\ every} \ t \in [0,1].
	\end{equation}
We divide the argument into two cases: 
	\begin{enumerate}
		\item[(I)] 
		$\wt{\eta} (t , \theta,u) \not \in \wt{N}_{\frac{2\rho}{3}} ( \wt{K}_c ) $ for every $ t \in [0,1]$. 
		\item[(II)]
		There exists a $t_0 \in [0,1]$ such that $\wt{\eta}(t_0 , \theta,u) \in \wt{N}_{\frac{2\rho}{3}} ( \wt{K}_c )$. 
	\end{enumerate}

	In Case (I), it follows from \eqref{eq:2.22} and \eqref{eq:2.23} that 
$\wt{\eta} (t) = \wt{\eta} (t,\theta,u) \in \wt{K}_{c,\delta_\rho,\rho/3}$ and 
$ \| \rd J ( \wt{\eta} (t) ) \|_{\wt{\eta}(t),\ast} \geq \delta_\rho$ for every $ t \in [0,1]$. 
Therefore, by $w(\theta, u) \leq 2 \| J(\theta, u) \|_{(\theta,u),\ast} $ 
due to \eqref{eq:2.17} and \eqref{eq:2.16}, 
we get 
	\[
		\frac{\rd}{\rd t} J ( \wt{\eta} (t ) ) = - \rd J( \wt{\eta} ( t ) ) 
		\left[  \frac{W(\wt{\eta}(t))} {  w( \wt{\eta}(t) ) } \right] 
		\leq -  \frac{1}{2} \| \rd J ( \wt{\eta} (t) ) \|_{ \wt{\eta} (t) , \ast } 
		\leq - \frac{1}{2} \delta_\rho
	\]
and by \eqref{eq:2.22}, one has 
	\[
		J ( \wt{\eta} (1) ) \leq J ( \wt{\eta} (0 , \theta,u) ) - \frac{1}{2} \delta_\rho 
		\leq c + \e - \frac{1}{2} \delta_\rho < c - \e.
	\]
This contradicts \eqref{eq:2.23} and Case (I) does not occur.

	In Case (II), if there exists a $t_1 \in [0,1]$ such that 
$\wt{\eta} (t_1 , \theta,u) \in \partial \wt{N}_\rho ( \wt{K}_c )$, 
then we may find $[t_2,t_3] \subset [0,1]$ so that 
	\begin{equation}\label{eq:2.24}
		\begin{aligned}
			&\wt{\eta} ( (t_2,t_3) , \theta, u ) \subset \wt{N}_\rho ( \wt{K}_c ) \setminus \wt{N}_{\frac{2\rho}{3}} (\wt{K}_c) 
			\subset \wt{K}_{c,\delta_\rho, \frac{\rho}{3}}, 
			\\
			&\text{either }
			(\wt{\eta} (t_2), \wt{\eta} (t_3) ) 
			\in \partial  \wt{N}_\rho ( \wt{K}_c ) \times \partial \wt{N}_{\frac{2\rho}{3}} (\wt{K}_c)
			\text{ or else } 
			(\wt{\eta} (t_2), \wt{\eta} (t_3) ) \in \partial  \wt{N}_{\frac{2\rho}{3}} ( \wt{K}_c ) \times \partial \wt{N}_{\rho} (\wt{K}_c).
		\end{aligned}
	\end{equation}
Since $\| \rd \wt{\eta} (t) / \rd t  \|_{\wt{\eta} (t) } \leq 3/2$ and it follows from 
\eqref{eq:2.20}, \eqref{eq:2.22}, \eqref{eq:2.23} and \eqref{eq:2.24} that 
	\[
		\frac{3}{2} \left( t_3-t_2 \right)  
		\geq \int_{t_2}^{t_3} \left\| \frac{\rd \wt{\eta}}{\rd t} (\tau) \right\|_{\wt{\eta} (\tau)} \rd \tau 
		\geq d_0 \left( \wt{\eta}(t_2) , \wt{\eta} (t_3) \right) \geq \frac{\rho}{3},
		\quad \frac{\rd}{\rd t} J (\wt{\eta} (t)) \leq - \frac{1}{2} \delta_\rho \quad 
		{\rm for\ all} \ t \in [t_2,t_3],
	\]
we obtain 
	\[
		\begin{aligned}
			J( \wt{\eta} (1) ) \leq J ( \wt{\eta} (t_3) ) 
			&= J( \wt{\eta} (t_2) ) + \int_{t_2}^{t_3} \frac{\rd}{\rd \tau} J ( \wt{\eta} (\tau) ) \rd \tau 
			\\
			&\leq J(\wt{\eta} (t_2)) - \int_{t_2}^{t_3} \frac{1}{2} \delta_\rho \rd \tau 
			\leq c + \e - \frac{\rho \delta_\rho}{9} < c - \e.
		\end{aligned}
	\]
This is a contradiction. From the above argument, 
we infer that $\wt{\eta} (t , \theta, u) \in \wt{N}_\rho ( \wt{K}_c )$ for all $t \in [0,1]$ 
when $(\theta,u) \in [J \leq c + \e]$ and $\wt{\eta} (1,\theta,u) \not\in [J \leq c - \e]$. 
Hence, we conclude that $\wt{\eta} (1, [J \leq c + \e] ) \subset [J \leq c - \e] \cup \wt{N}_\rho(\wt{K}_c)$.

	When $(\theta, u) \in [J \leq c + \e] \setminus \wt{N}_\rho( \wt{K}_c )$, 
if $\wt{\eta} (1, \theta, u) \not \in [ J \leq c - \e]$, then 
by (iv), one sees $(\theta,u) \in \wt{K}_{c,\delta_\rho,\rho/3}$. 
As in the above, we may see that Case (I) does not occur and 
$\wt{\eta} ([0,1], (\theta,u)) \cap \wt{N}_{2\rho/3} ( \wt{K}_c ) \neq \emptyset$. 
Thus, \eqref{eq:2.24} holds, however, again we get a contradiction $J(\wt{\eta} (1)) < c - \e$. 
Hence, $\wt{\eta} (1, [ J \leq c + \e ] \setminus \wt{N}_\rho ( \wt{K}_c )  ) \subset [J \leq c - \e]$. 
Therefore, Lemma \ref{lemma:2.9} holds when $\wt{K}_c \neq \emptyset$.

	Secondly, we deal with the case $\wt{K}_c = \emptyset$, 
and let $\delta_0>0$ be a number in Proposition \ref{proposition:2.8} (i) and 
$\be > 0$ satisfy $0 < \be < \delta_0$. Then instead of \eqref{eq:2.21}, we shall consider 
	\begin{equation}\label{eq:2.25}
		\frac{\rd \wt{\eta}}{\rd t} 
		= -  \psi \left( J( \wt{\eta} (t) ) \right) \frac{W(\wt{\eta}(t))}
		{w ( \wt{\eta} (t) ) }, \quad 
		\wt{\eta} (0,\theta,u) = (\theta, u). 
	\end{equation}
Select an $\e > 0$ so that $\e < \min \{  \be/2,  \delta_0 / 4 \} $. Since it is easy to check (i)--(iv), 
we shall prove $\wt{\eta} (1, [J \leq c + \e] ) \subset [J \leq c - \e]$. 
If $(\theta,u) \in [J \leq c + \e]$ and $\wt{\eta} (1,\theta, u) \not \in [J \leq c - \e]$, then 
$ c - \e < J( \wt{\eta}(t,\theta,u) ) \leq c + \e$ for all $ t \in [0,1]$. 
Thus, \eqref{eq:2.25}, Proposition \ref{proposition:2.8} (i), \eqref{eq:2.16} and  \eqref{eq:2.17} imply 
	\[
		\frac{\rd}{\rd t} J (\wt{\eta} (t,\theta, u) ) \leq - \frac{1}{2}  \| \rd J( \wt{\eta} (t,\theta,u) ) \|_{\wt{\eta}(t)}  
		\leq - \frac{1}{2} \delta_0.
	\]
Hence, we get a contradiction:
	\[
		c + \e \leq J ( \wt{\eta} (1,\theta,u) ) 
		\leq J ( \wt{\eta} (0,\theta,u) ) - \frac{1}{2} \delta_0 \leq c + \e - \frac{1}{2} \delta_0 
		< c - \e .
	\]
Thus, (v) holds and we complete the proof. 
	\end{proof}

	\begin{remark}\label{remark:2.10}
		By the proof of Lemma \ref{lemma:2.9}, 
		if $0< \inf \{ \| \rd J(\theta, u) \|_{(\theta,u),\ast} \ |\ |J(\theta,u)- c| < \be \}$ for some $\be>0$, 
		then $\wt{K}_c = \emptyset$ and 
		the assertions in Lemma \ref{lemma:2.9} still hold. 
	\end{remark}

To prove the deformation lemma for $I$ from $\wt{\eta}$, we define 
	\[
		i (u) := (0,u) : X \to \R \times X, \quad \pi (\theta, u) := T_\theta u : \R \times X \to X.
	\]
Note that $J(\theta, u) = I(\pi (\theta, u))$.

	\begin{lemma}\label{lemma:2.11}
		For every $\rho \in (0,1)$ there exists an $R(\rho) > 0$ such that 
			\[
				\pi \left( \wt{N}_\rho ( \wt{K}_c  ) \right) \subset N_{R(\rho)} (K_c), \quad 
				i \left( X \setminus N_{R(\rho)} (K_c) \right) 
				\subset \left( \R \times X \right) \setminus \wt{N}_\rho ( \wt{K}_c ), 
				\quad \lim_{\rho \to 0} R(\rho) = 0
			\]
		where $N_\alpha (A) := \set{ u \in X | \| u - A \|_X < \alpha }$. 
	\end{lemma}

	\begin{proof}
Let $\rho \in (0,1)$ and $(\theta, u) \in \wt{N}_\rho (\wt{K}_c)$. By \eqref{eq:2.5} and \eqref{eq:2.13}, 
remark that 
	\[
		\rho > d_0 \left( (\theta, u) , \wt{K}_c \right) 
		= d_0 \left( (0, T_\theta u ), \wt{K}_c \right). 
	\]
Choose $\sigma \in C^1([0,1], \R \times X)$ so that 
	\[
		\sigma (0) = (0,T_\theta u), \quad 
		\sigma (1) \in \wt{K}_c, \quad 
		\int_0^1 \| \dot{\sigma} (\tau) \|_{\sigma (\tau)} \rd \tau \leq 2 \rho. 
	\]
Writing $\sigma (\tau) = ( \sigma_\theta (\tau), \sigma_u (\tau) )$, 
we observe from $\sigma_\theta(0) = 0$ that 
	\[
		| \sigma_\theta (\tau) | \leq \int_0^\tau | \dot{\sigma}_\theta (t) | \rd t 
		\leq \int_0^1 \| \dot{\sigma} (t) \|_{\sigma(t)} \rd t \leq 2 \rho 
		\quad {\rm for\ all} \ \tau \in [0,1].
	\]
Next, by \eqref{eq:2.1}, there exists a $c_1 > 0$ such that 
	\[
		c_1 \| v \|_X \leq \| T_\Theta v \|_X \quad 
		{\rm for\ all} \ (\Theta, v) \in [-2\rho, 2 \rho] \times X.
	\]
Then 
	\[
		\| T_\theta u - \sigma_u(1) \|_X 
		\leq \int_0^1 \| \dot{\sigma}_u (t) \|_X \rd t 
		\leq c_1^{-1} \int_0^1 \| \dot{\sigma} (t) \|_{\sigma (t)} \rd t 
		\leq 2 c_1^{-1} \rho.
	\]
Since $( \sigma_\theta (1) , \sigma_u(1) ) \in \wt{K}_c$ and 
$T_{\sigma_\theta(1)} \sigma_u(1) \in K_c$ due to \eqref{eq:2.13}, one has 
	\[
		 d_X ( T_\theta u, K_c ) \leq 
		\| T_\theta u - T_{\sigma_\theta (1)} \sigma_u(1) \|_X 
		\leq \| T_\theta u - \sigma_u(1) \|_X + \| \sigma_u(1) - T_{\sigma_\theta (1)} \sigma_u(1) \|_X.
	\]
Therefore, from $\pi(\theta, u) = T_\theta u$ and $T_{\sigma_\theta (1)} \sigma_u(1) \in K_c$,  
it follows that 
	\[
		\begin{aligned}
			d_X ( \pi(\theta, u), K_c ) 
			&\leq 2 \rho c_1^{-1} 
			+ \sup \left\{ \| T_{\Theta} w - w \|_X 
			\;\middle|\; 
			| \Theta | \leq 2 \rho, \ w \in K_c \right\} 
			\\
			&< 3 \rho c_1^{-1} 
			+ \sup \left\{ \| T_{\Theta} w -  w \|_X \;\middle|\;
			| \Theta | \leq 2 \rho, \ w \in K_c \right\} =: R(\rho).
		\end{aligned}
	\]
Note that by \PSPc, $K_c$ is compact. Thus, \eqref{T1} yields $R(\rho) \to 0$ as $\rho \to 0$. 
Since $(\theta, u) \in \wt{N}_\rho( \wt{K}_c )$ is arbitrary, we obtain 
$\pi ( \wt{N}_\rho ( \wt{K}_c ) ) \subset N_{R(\rho)} (K_c)$.

	On the other hand, if $u \in X \setminus N_{R(\rho)} (K_c)$ and 
$i(u) = (0,u) \in \wt{N}_\rho (\wt{K}_c)$, 
then we have a contradiction $u = \pi (i (u) ) \in N_{R(\rho)} (K_c)$. 
Therefore, $i (X \setminus N_{R(\rho) } (K_c ) ) \subset \left(  \R \times X \right) \setminus \wt{N}_\rho (\wt{K}_c)$ and 
this completes the proof. 
	\end{proof}

Now we prove Lemma \ref{lemma:2.5}.

	\begin{proof}[Proof of Lemma \ref{lemma:2.5}]
Let $\be >0$, $K_c \neq \emptyset$ and $\cO \supset K_c$ be any open set. 
Since $K_c$ is compact by \PSPc, choose $\rho>0$ sufficiently small so that $N_{R(\rho)} (K_c) \subset \cO$. 
Let $\e \in (0,\be)$ and $\wt{\eta}$ be constructed in Lemma \ref{lemma:2.9}. 
Then we define $\eta$ by 
	\[
		\eta (t,u) := \pi \left( \wt{\eta} ( t, i(u) ) \right) 
		= \pi \left( \wt{\eta} ( t, 0, u) \right) \in C( [0,1] \times X, X ).
	\]
By $ I (\eta (t,u)) = J( \wt{\eta} (t,0,u) )$, 
it is immediate to check (i)--(iv). For (v), 
since $[I \leq c + \e] \setminus \cO \subset [ I \leq c + \e] \setminus N_{R(\rho)} (K_c)$ and 
$i( [I \leq c + \e] ) \subset [ J \leq c + \e]$, we obtain 
$i ( [I \leq c + \e] \setminus \cO ) \subset [J \leq c + \e] \setminus \wt{N}_\rho (\wt{K}_c)$ 
thanks to Lemma \ref{lemma:2.11}. 
Thus, Lemma \ref{lemma:2.9} gives 
	\[
		\wt{\eta} \left( 1 , i \left( [I \leq c + \e] \setminus \cO \right) \right) 
		\subset \wt{\eta} \left( 1 , [ J \leq c + \e] \setminus \wt{N}_\rho ( \wt{K}_c ) \right) 
		\subset [ J \leq c - \e].
	\]
From $\pi ( [ J \leq c - \e] ) = [I \leq c - \e]$, it follows that 
$\eta ( 1 , [ I \leq c + \e ] \setminus \cO ) \subset [ I \leq c - \e]$.

	In a similar way, Lemma \ref{lemma:2.9} yields 
$\wt{\eta} (1, i( [I \leq c + \e] ) ) \subset \wt{\eta} (1, [J \leq c + \e]  ) 
\subset [ J \leq c - \e] \cup \wt{N}_\rho ( \wt{K}_\rho ) $. 
Thus, 
Lemma \ref{lemma:2.11} implies 
$\eta (1, [I \leq c + \e] ) \subset [I \leq c - \e] \cup \wt{N}_{R(\rho)} (\wt{K}_c)
\subset [ I \leq c - \e] \cup \cO$.

	When $K_c = \emptyset$, since $\wt{K}_c = \emptyset$, 
Lemma \ref{lemma:2.9} gives $\wt{\eta} (1, [J \leq c + \e] ) \subset [J \leq c - \e]$. 
Hence, $\eta (1, [I \leq c + \e]) \subset [I \leq c - \e]$ and 
we complete the proof. 
	\end{proof}

	\begin{remark}\label{remark:2.12}
By \eqref{eq:2.12} and Remark \ref{remark:2.10}, 
if $0< \inf \{ \| \rd J (0,u) \|_{(\R \times X)^\ast} \ |\ | J(0,u) - c | < \be \}$ for some $\be > 0$, 
then the assertions of Lemma \ref{lemma:2.5} hold. 
	\end{remark}


\subsection{Proof of Theorems \ref{theorem:2.3} and \ref{theorem:2.4}}
\label{section:2.2}


	We first treat Theorem \ref{theorem:2.3}:
	
	\begin{proof}[Proof of Theorem \ref{theorem:2.3}]
Notice that under the assumption of Theorem \ref{theorem:2.3}, 
$c_{mp}$ is well-defined and $c_{mp} > 0$ holds. 
In addition, the existence of $(u_n)$ in Theorem \ref{theorem:2.3} is 
equivalent to 
	\begin{equation}\label{eq:2.26}
		0= \inf \left\{ \left\| \rd J(0,u) \right\|_{(\R \times X)^\ast} \ \Big| \ 
		\left| J(0,u) - c_{mp} \right| < \frac{1}{n} \right\}
	\end{equation}
for each $n \geq 1$. Thus, by Remark \ref{remark:2.12} and the standard argument 
(see Rabinowitz \cite{R}, for instance), 
we have \eqref{eq:2.26} and the existence of $(u_n)$ follows. 
In addition, if $I$ satisfies $\mathrm{(PSP)}_{c_{mp}}$, then 
we may find a $u_0 \in X$ such that $\rd I(u_0) = 0$ with 
$\rd_{\theta} J (0,u_0) = 0$. 
Therefore, Theorem \ref{theorem:2.3} holds. 
	\end{proof}

	Next, we shall prove Theorem \ref{theorem:2.4}.

	\begin{proof}[Proof of Theorem \ref{theorem:2.4}]
We first define the minimax values: 
	\begin{equation}\label{eq:2.27}
		\begin{aligned}
			c_j &:= \inf_{\gamma \in \Gamma_j } \max_{\sigma \in D_j} 
			I (\gamma (\sigma)), \quad 
			D_j := \left\{\sigma \in \R^j \;\middle|\; |\sigma | \leq 1\right\},
			\\
			\Gamma_j &:= \Set{ \gamma \in C( D_j, X ) |
			\gamma(-\sigma) = - \gamma (\sigma) \ {\rm for}\ \sigma \in D_j, \ 
			\gamma = \gamma_j \ {\rm on} \ \partial D_j = S^{j-1} },
			\\
			d_j &:= \inf_{A\in \Lambda_j} \max_{ u \in A} I(u), \\
			\Lambda_j &:= \Set{ \gamma \left( \ov{D_m \setminus B } \right) 
			| 
			m \geq j, \ \gamma \in \Gamma_{m}, \ 
			B \subset D_m \setminus \{0\}: \text{closed}, 
			-B = B, \ g(B) \leq m-j   }
		\end{aligned}
	\end{equation}
where $g(B)$ denotes the Krasnoselskii genus of $B$. Remark that 
$d_j \leq c_j$ holds due to $g( \emptyset ) = 0$. 
Furthermore, $\Gamma_j \neq \emptyset$ since a map defined by
$\gamma_{0,j} (\sigma) := |\sigma| \gamma_j(\sigma / |\sigma|)$ for $\sigma \in D_j \setminus \{0\}$ and 
$\gamma_{0,j} (0) := 0$ belongs to $\Gamma_j$.

	Next, as in \cite[section 9]{R} and \cite[Lemma 2.4]{HT-18}, 
noting that $\gamma^{-1} ( \ov{B_{\rho_0}(0)} ) \cap S^{j-1} = \emptyset$ 
for any $\gamma \in \Gamma_j$ due to 
$\gamma_j (S^{j-1}) \cap \ov{B_{\rho_0}(0)} = \emptyset$ and $\gamma = \gamma_j$ on $S^{j-1}$, 
we may prove that 
	\begin{itemize}
		\item $\Lambda_{j+1} \subset \Lambda_j$ and $d_j \leq d_{j+1}$. 
		\item If $\psi \in C(X,X)$ is odd and satisfies $\psi (u) = u$ provided $u \in [I \leq 0]$, then 
		$\psi \circ \gamma \in \Gamma_j$ for any $\gamma \in \Gamma_j$, hence, 
		$\psi (A) \in \Lambda_j$ for each $A \in \Lambda_j$. 
		\item For every $A \in \Lambda_j$ and closed set $Z \subset X \setminus \{0\}$ with $-Z=Z$ 
		and $ g(Z) \leq s < j$, we have $\ov{A \setminus Z } \in \Lambda_{j-s}$. 
		\item $A \cap \partial B_{\rho_0} (0) \neq \emptyset$ for any $A \in \Lambda_j$. 
		Thus, Theorem \ref{theorem:2.4} (i) yields 
		$0 < \inf_{\| u \|_{X} = \rho_0} I(u) \leq d_j \leq c_j$. 
	\end{itemize}

	Next, we claim that $K_{c_j} \neq \emptyset$ and $K_{d_j} \neq \emptyset$. 
Indeed, if $K_{d_j} = \emptyset$, then we apply Lemma \ref{lemma:2.5} with $c = d_j$. 
Let $\be \in (0, d_j/2)$, $\eta$ and $\e \in (0,\be)$ appear in Lemma \ref{lemma:2.5}. 
Choose $A \in \Lambda_j$ so that $\max_{ u \in A} I(u) \leq d_j + \e$. 
By the properties of $\eta$, we see $\eta (1, A ) \in \Lambda_j$ and 
$\max_{u \in \eta(1,A)} I(u) \leq d_j - \e$, which is a contradiction. 
Hence, $K_{d_j} \neq \emptyset$. In a similar way, we can prove $K_{c_j} \neq \emptyset$.

	Since $\rd J(0,u) = 0$ and $I(u) = c_j \geq d_j$ for $u \in K_{c_j}$, 
it suffices to show that $d_j \to \infty$ as $j \to \infty$. 
Following \cite[section 9]{R} and replacing $\mathcal{K}$ in \cite[proof of Proposition 9.33]{R} 
by $\mathcal{K}:= \set{ u \in X | d_{1} \leq I(u) \leq \ov{d}, \ \rd J(0,u) = 0 }$ where 
$\ov{d} := \lim_{j \to \infty} d_j$, we may also show that 
	\begin{itemize}
		\item 
			If $d_{j}= d_{j+1} = \cdots = d_{j+\ell} =: d$, then $g( K_d ) \geq \ell + 1$.
		\item 
			$d_j \to \infty$ as $j \to \infty$. 
	\end{itemize}
This completes the proof of Theorem \ref{theorem:2.4}. 
\end{proof}


\section{Proof of Theorems \ref{theorem:1.1} and \ref{theorem:1.3}}\label{section:3}


	In this section, using Theorems \ref{theorem:2.3} and \ref{theorem:2.4}, 
we prove Theorems \ref{theorem:1.1} and \ref{theorem:1.3}. 
We first deal with Theorem \ref{theorem:1.1}. 
In order to apply Theorem \ref{theorem:2.4}, 
we modify the nonlinearity $f(t)$. As the first step, we prove

	\begin{lemma}\label{lemma:3.1}
		Suppose that $f$ satisfies \textup{(f0)}, \textup{(f1)} and \textup{(f3)}. 
		Assume also that there exists a $\zeta_2 > \zeta_1$ such that $f(t) = 0$ for each $t \geq \zeta_2$. 
		Then $\| u \|_{L^\infty(\RN)} \leq \zeta_2$ holds for every solution $u$ of \eqref{eq:1.1}.
	\end{lemma}

	\begin{proof}
Consider $v_0(x) := \max\{ u(x) , \zeta_2 \} - \zeta_2 = (u(x) - \zeta_2 )_+ $ 
where $a_+ := \max\{ 0, a \}$ and 
let $\varphi_1 \in C^\infty_0(\RN)$ be a cut off function with 
$0 \leq \varphi_1 (x) \leq 1$, $\varphi_1(x)=0$ if $|x| \geq 2$ and 
$\varphi_1(x) =1$ if $|x| \leq 1$. 
Set also $\varphi_R(x) := \varphi_1(x/R)$ for $R \geq 1$. 
Then we can check that 
$v_0 \in \cDs{\RN}$ and $\| \varphi_R v_0 - v_0 \|_{\cDs{\RN}} \to 0$ as $R \to \infty$. 
Since $u$ is a solution of \eqref{eq:1.1} and $\varphi_R v_0$ can be approximated by 
functions in $C^\infty_0(\RN)$ and $f(u) \in L^\infty(\R)$ due to the assumption, we have 
	\begin{equation}\label{eq:3.1}
		\la u , \varphi_R v_0 \ra_{\cDs{\RN} } 
		= \int_{\RN} f( u(x) ) \varphi_R v_0(x) \rd x \quad {\rm for\ all} \ R \geq 1. 
	\end{equation}
By the definition of $v_0$ and the assumption on $f$, 
we observe that $f(u(x)) \varphi_R (x) v_0(x) \equiv 0$ on $\RN$. 
Letting $R \to \infty$ in \eqref{eq:3.1}, we obtain 
	\[
		0= \la u, v_0 \ra_{\cDs{\RN} }
		= \int_{\RN \times \RN} \frac{( u(x) - u(y) ) ( v_0(x) - v_0 (y) ) }{|x-y|^{N+2s}} \rd x \rd y 
		\geq \| v_0 \|^2_{\cDs{\RN}},
	\]
which implies $v_0 \equiv 0$ and $ u(x) \leq \zeta_2$.

	In a similar way, we may prove $ - \zeta_2 \leq u(x)$ and Lemma \ref{lemma:3.1} holds. 
	\end{proof}

	By Lemma \ref{lemma:3.1}, without loss of generality, 
we may assume the following condition instead of (f2):
	\begin{equation}
		\lim_{t \to \infty} \frac{f(t)}{t^{2^\ast_s-1}} = 0.
		\tag{f2'}
	\end{equation}
In fact, if 
	\begin{equation}\label{eq:3.2}
		f(t) > 0 \quad {\rm for\ all} \ t > \zeta_1,
	\end{equation}
then (f2) and \eqref{eq:3.2} imply (f2'). 
On the other hand, 
if there exists a $\zeta_2 > \zeta_1$ such that 
$f(\zeta_2) = 0$, then set 
	\[
		\ov{f} (t) := \left\{\begin{aligned}
			& f(t) & &{\rm if} \ 0 \leq t \leq \zeta_2, \\
			& 0 & &{\rm if} \ \zeta_2 < t
		\end{aligned}\right.
	\]
and extend $\ov{f}$ as an odd function on $\R$. 
Remark that $\ov{f}$ satisfies (f2') and 
instead of \eqref{eq:1.1}, we consider 
	\begin{equation}\label{eq:3.3}
		(-\Delta)^s u = \ov{f} (u) \quad {\rm in} \ \RN, \quad u \in \cDs{\RN}.
	\end{equation}
By Lemma \ref{lemma:3.1}, any solution $w$ of \eqref{eq:3.3} also satisfies \eqref{eq:1.1} and 
$I_0(w) = \ov{I}_0 (w)$ where 
	\[
		\ov{I}_0 (u) := \frac{1}{2} \|u\|_{\cDs{\RN}}^2 - \int_{\RN} \ov{F} (u) \rd x, 
		\quad 
		\ov{F} (t) := \int_0^t \ov{f} (\tau) \rd \tau
		.
	\]
Therefore, instead of $f$, we may use $\ov{f}$ to obtain the desired solutions of \eqref{eq:1.1}.

	In what follows, we assume that (f0), (f1), (f2') and (f3) hold. 
By (f3), we have the following two cases: 
	\begin{enumerate}
		\item[(I)] There exists a $\xi_0 \in (0,\zeta_1)$ such that $f(\xi_0) = 0$. 
		\item[(II)] $f(t) > 0$ for each $ t \in (0, \zeta_1)$. 
	\end{enumerate}

	In case (I), writing $f(t) = f_+(t) - f_-(t)$ where 
$f_{\pm} (t) := \max\{ \pm f(t) , 0 \}$ if $ t \geq 0$ and 
$f_\pm (t) := - f_\pm (-t)$ if $t < 0$, 
for each $\e \in (0,1]$ and $t \geq 0$, set 
	\[
		f_{\e, -} (t) := \left\{\begin{aligned}
			& \min \left\{  f_-(t) , \e^{-1} t^{2^\ast_s - 1}  \right\} 
			& &{\rm if} \ 0 \leq t \leq \xi_0,
			\\
			& f_- (t) & &{\rm if} \ \xi_0 < t
		\end{aligned}\right., \quad 
		f_{\e} (t) := f_+(t) - f_{\e,-}(t)
	\]
and extend $f_{\e,-}, f_\e$ as odd functions on $\R$. It is easily seen that $f_{\e,-}, f_\e \in C(\R)$ and 
	\begin{align}
		& f_{\e_2,-} (t) \leq f_{\e_1,-}(t) \leq f_-(t) \quad \text{for each $0< \e_1 < \e_2 \leq 1$ and $ t \geq 0$},
		\label{align:3.4}
		\\
		&\sup_{t \in \R} \left\{ \left| f_{\e,-}(t) - f_-(t) \right| 
		+ \left| f_\e(t) - f(t) \right| \right\}  \to 0 \quad {\rm as} \ \e \to 0.
		\label{align:3.5}
	\end{align}

		On the other hand, in case (II), we do not need any further modification of $f$ and 
for notational convenience, write $f_{\e,-}(t) := f_-(t)$ and $f_\e(t) := f_+(t) - f_{\e,-} (t) = f(t)$ for each $\e \in (0,1]$. 
Then \eqref{align:3.4} and \eqref{align:3.5} still hold in this case.

	In either case, we remark that $f_\e$ satisfies $f(t) \leq f_\e(t)$ for $t \geq 0$ and 
	\begin{equation}\label{eq:3.6}
		\text{(f0), (f2'), (f3), }
		\lim_{t \to +0} \frac{f_{+} (t) }{t^{2^\ast_s - 1}} = 0, \quad 
		0 \leq \liminf_{t \to +0} \frac{f_{\e,-}(t)}{t^{2^\ast_s - 1 }} 
		\leq \limsup_{t \to +0} \frac{ f_{\e,-} (t) }{t^{2^\ast_s - 1}} \leq \e^{-1} < \infty. 
	\end{equation}
Therefore, the functional defined by 
	\[
		I_\e (u) := \frac{1}{2} \| u \|_{\cDs{\RN}}^2 - \int_{\RN} F_+(u) - F_{\e,-}(u) \rd x 
		= \frac{1}{2} \| u \|_{\cDs{\RN}}^2 - \int_{\RN} F_+ (u) \rd x + \int_{\RN} F_{\e,-} (u) \rd x 
	\]
enjoys $I_\e \in C^1( \cDs{\RN}, \R )$ and 
	\begin{equation}\label{eq:3.7}
		I_{\e_2} (u) \leq I_{\e_1} (u) \leq I_0(u) \quad {\rm for\ every} \ 0<\e_1<\e_2 \leq 1, \ 
		u \in \cDs{\RN}.
	\end{equation}

	We shall apply Theorem \ref{theorem:2.4} for 
	\[
		X = \cDsr =: \Set{ u \in \cDs{\RN} | u(x) = u(|x|) \ {\rm for \ a.a.}\ x \in \RN },
	\]
$I=I_\e$, $T_\theta u (x) := u(e^{-\theta} x)$ and 
	\[
		\begin{aligned}
			J_\e(\theta,u) := I_\e (T_\theta u) 
			&= \frac{e^{(N-2s)\theta}}{2} \|u\|_{\cDs{\RN}}^2 - e^{N \theta} \int_{\RN} F_+(u) \rd x 
			+ e^{N\theta} \int_{\RN} F_{\e,-} (u) \rd x 
			\\
			&\in C^1(\R \times \cDsr , \R).
		\end{aligned}
	\]
As pointed out in section \ref{section:2}, it is immediate to check that 
$X,I_\e,T_\theta, J_\e$ satisfy \eqref{T1} and \eqref{T2}. 
What remains to check is \PSPc \  and the symmetric mountain pass structure 
for $(X,I_\e,J_\e)$.

	\begin{lemma}\label{lemma:3.2}
		For each $\e \in (0,1]$, 
		$(X,I_\e)$ satisfies 
		the conditions \emph{(i)} and \emph{(ii)} in Theorem \ref{theorem:2.4}. 
	\end{lemma}

	\begin{proof}
We first show (i). Remark that 
	\[
		I_\e(u) \geq \frac{1}{2}\|u\|_{\cDs{\RN}}^2 - \int_{\RN} F_+(u) \rd x. 
	\]
From \eqref{eq:3.6} we can find a $C>0$ so that 
	\[
		|F_+(t) |  \leq C |t|^{2^\ast_s} \quad \text{for all $t \in \R$}.
	\]
Therefore, Sobolev's inequality yields 
	\[
		I_\e(u) \geq \frac{1}{2} \|u\|_{\cDs{\RN}}^2 - C \| u \|_{L^{2^\ast_s}(\RN)}^{2^\ast_s} 
		\geq \frac{1}{2} \|u\|_{\cDs{\RN}}^2 - C \|u\|_{\cDs{\RN}}^{2^\ast_s}.
	\]
Since $2^\ast_s>2$, there exists a $\rho_0>0$ such that 
	\begin{equation}\label{eq:3.8}
		\inf_{ \| u \|_{\cDs{\RN}} = \rho_0}  I_\e (u) \geq  \inf_{ \| u \|_{\cDs{\RN}} = \rho_0} 
		\rho_0^2 \left( \frac{1}{2} - C \rho_0^{2^\ast_s - 2} \right) > 0 , \quad 
		\| u \|_{\cDs{\RN}} \leq \rho_0 \quad \Rightarrow \quad 
		I_\e (u) \geq 0.
	\end{equation}

	Next we treat (ii). 
We first remark that following Berestycki and Lions \cite[sections 9.2a and 9.2b]{BL-83-2},
under (f0) and (f3), for each $k \geq 1$ 
we may find $R_k > 0$ and $\wt{\gamma}_k \in C( S^{k-1} , H^1_{\rm rad} (\RN) )$ such that 
	\[
		\wt{\gamma}_k(-\sigma) = - \wt{\gamma}_k( \sigma ), \quad 
		{\rm supp}\, \wt{\gamma}_k(\sigma) \subset B_{R_k}(0), \quad 
		1 \leq \int_{\RN} F( \wt{\gamma}_k(\sigma) ) \rd \sigma < \infty 
		\quad \text{for any $\sigma \in S^{k-1}$}.
	\]
By $H^1_{\rm rad}(\RN) \subset \cDsr$ and \eqref{eq:3.7}, one sees that 
	\[
		\begin{aligned}
			I_\e \left( \wt{\gamma}_k(\sigma) (e^{-\theta} \cdot ) \right) 
			\leq I_0 \left( \wt{\gamma}_k(\sigma) (e^{-\theta} \cdot ) \right) 
			&= \frac{e^{(N-2s)\theta}}{2} \left\|\wt{\gamma}_k(\sigma) \right\|^2_{\cDs{\RN}} 
			- e^{N\theta} \int_{\RN} F \left( \wt{\gamma}_k (\sigma) \right) \rd x
			\\
			& \leq \frac{e^{(N-2s)\theta}}{2} 
			\max_{\sigma \in S^{k-1}} \left\|\wt{\gamma}_k(\sigma) \right\|^2_{\cDs{\RN}}  
			- e^{N\theta} .
		\end{aligned}
	\]
Hence, choosing sufficiently large $\theta_k>0$ and setting 
$\gamma_k(\sigma) (x) := \wt{\gamma}_k (\sigma) ( e^{-\theta_k} x )$, 
we get
	\[
		\max_{\sigma \in S^{k-1}} I_\e (\gamma_k (\sigma) ) 
		\leq \max_{\sigma \in S^{k-1}} I_0 ( \gamma_k (\sigma) )\leq - 1.
	\]
Furthermore, from \eqref{eq:3.8} it follows that 
	\[
		\gamma_k(S^{k-1}) \subset \cDsr \setminus \ov{B_{\rho_0} (0) }.
	\]
Thus we complete the proof. 
	\end{proof}

	\begin{remark}\label{remark:3.3}
As a byproduct of the proofs of Theorem \ref{theorem:2.4} and Lemma \ref{lemma:3.2}, 
since ${\rm supp} \, \gamma_k(\sigma) \subset B_{e^{\theta_k} R_k } (0)$ for each $\sigma \in S^{k-1}$,  
we obtain an upper bound for each $c_{\e,k}$:
	\[
		c_{\e,k} \leq \ov{c}_k := \max_{\sigma \in D^k} 
		I \left( \gamma_{0,k} (\sigma) \right) < \infty
	\]
where $c_{\e,k}$ is defined though \eqref{eq:2.27} with $I_\e$ and 
$\gamma_{0,k}(\sigma) := |\sigma| \gamma_k(\sigma/ |\sigma)$ if $|\sigma|>0$ and 
$\gamma_{0,k} (0) := 0$. 
We also remark that $c_{\e,k}$ is monotone due to \eqref{eq:3.7}: 
$c_{\e_2, k} \leq c_{\e_1,k} \leq \ov{c}_k$ for each $0<\e_1 < \e_2 \leq 1$. 
	\end{remark}

	\begin{lemma}\label{lemma:3.4}
		For each $\e \in (0,1]$ and $ c \in \R$, $I_\e$ satisfies \PSPc. 
	\end{lemma}

	\begin{proof}
Assume that $(u_n)_n \subset \cDsr$ satisfies 
$I_\e(u_n) \to c \in \R$ and $\| \rd J_\e(0,u_n) \|_{\R \times X^\ast} \to 0  $. From 
	\[
		N c + o(1) = N I_\e(u_n) - \rd_\theta J_\e (0,u_n) 
		= s \| u_n \|_{\cDs{\RN}}^2,  
	\]
we find that $(u_n)_n$ is bounded in $\cDs{\RN}$.

	Up to a subsequence, suppose $u_n \rightharpoonup u_0$ weakly in $\cDs{\RN}$ 
and choose a $C_0>0$ so that 
	\begin{equation}\label{eq:3.9}
		  \| u_0 \|_{L^{2^\ast_s} (\RN)} 
		 + \sup_{ n \geq 1 }  \| u_n \|_{L^{2^\ast_s} (\RN) } \leq C_0.
	\end{equation}
To prove a strong convergence, it suffices to verify $\|u_n\|_{\cDs{\RN}} \to \|u_0\|_{\cDs{\RN}}$. 
For this purpose, 
we first remark that $u_n \to u_0$ strongly in $L^p_{\rm loc} (\RN)$ for any $1 \leq p < 2^\ast_s$. 
By Strauss' lemma (see \cite{S-77,BL-83-1}) and (f2'), for all $\varphi \in C^\infty_0(\RN)$, 
it follows that
	\[
		\int_{\RN} f_{+} (u_n) \varphi \rd x \to \int_{\RN} f_+(u_0) \varphi \rd x, \quad 
		\int_{\RN} f_{\e,-} (u_n) \varphi \rd x \to \int_{\RN} f_{\e,-} (u_0) \varphi \rd x,
	\]
which implies $\rd I_\e (u_0) = 0$. In particular, 
	\begin{equation}\label{eq:3.10}
		\|u_0\|^2_{\cDs{\RN}} = \int_{\RN} f_+(u_0) u_0 \rd x - \int_{\RN} f_{\e,-} (u_0) u_0 \rd x.
	\end{equation}

	Next, we shall prove 
	\begin{equation}\label{eq:3.11}
		\lim_{n \to \infty} \int_{\RN} f_{+} (u_n) u_n \rd x = \int_{\RN} f_+(u_0) u_0 \rd x. 
	\end{equation}
By Strauss' lemma and (f2'), for any $R>0$, we get 
	\begin{equation}\label{eq:3.12}
		\lim_{n \to \infty} \int_{|x| \leq R}  f_+(u_n) u_n \rd x 
		= \int_{|x| \leq R} f_+(u_0) u_0 \rd x .
	\end{equation}

	Let $\eta > 0$. Since $f_+(u_0) u_0 \in L^1(\RN)$ thanks to \eqref{eq:3.6}, choose an $R_\eta > 0$ so that 
	\begin{equation}\label{eq:3.13}
		\int_{|x| > R_\eta} |f_+(u_0) u_0| \rd x < \eta_.
	\end{equation}
Next, from \eqref{eq:3.6}, there exist $ 0 < t_\eta < T_\eta < \infty$ so that 
	\[
		|f_+(t) t | \leq \eta t^{2^\ast_s} \quad \text{if $|t| \leq 2t_\eta$ or $|t| \geq T_\eta$  }.
	\]
Hence, \eqref{eq:3.9} implies 
	\begin{equation}\label{eq:3.14}
		\int_{[|u_0| \leq 2 t_\eta] \cap [ |u_0| \geq T_\eta ] } |f_+(u_0) u_0 | \rd x 
		+ \sup_{ n \geq 1 } 
			\int_{[|u_n| \leq 2 t_\eta] \cap [ |u_n| \geq T_\eta ] } |f_+(u_n) u_n | \rd x 
			\leq 2 C_0^{2^\ast_s} \eta .
	\end{equation}

	Set $v_n(x) := \chi_\eta (u_n(x))$ where 
$ \chi_\eta \in C^\infty_0(\R)$ with 
	\[
		\chi_\eta(-t) = - \chi_\eta(t),  \quad |\chi_\eta(t)| \leq |t|, \quad 
		\chi_\eta(t) = 0 \quad \text{if $|t| \leq t_\eta$ or $|t| \geq 2 T_\eta$}, 
		\quad \chi_\eta (t) = t \quad \text{if $2 t_\eta \leq |t| \leq  T_\eta$}.
	\]
Noting that there exists a $C_\eta > 0$ so that 
	\[
		| v_n(x) |^2 \leq C_\eta |u_n(x)|^{2^\ast_s}, \quad 
		| v_n(x) - v_n(y) | \leq \| \chi_\eta' \|_{L^\infty(\R)} | u_n(x) - u_n(y) |
	\]
for $x, y \in \RN$, one sees that 
$(v_n)_n$ is bounded in $H^s_{\rm rad}(\RN)$. 
From the pointwise convergence, we may assume that 
$v_n \rightharpoonup v_0(x) =: \chi_\eta (u_0(x))$ weakly in $H^s_{\rm rad} (\RN)$. 
Next, fix $p_0 \in (2,2^\ast_s)$ and $C_{\eta,p_0}$ such that 
	\[
		\left| f_+(t)t \right| \leq C_{\eta,p_0} |t|^{p_0} \quad 
		\text{for all $t_\eta \leq |t| \leq 2 T_\eta$}.
	\]
Choosing $R_{\eta,p_0} \geq R_\eta$, we obtain
	\[
		\int_{  |x| \geq R_{\eta,p_0}  } |v_0(x)|^{p_0} \rd x \leq \frac{\eta}{C_{\eta,p_0}}. 
	\]
Since $H^s_{\rm rad} (\RN) \subset L^p(\RN)$ is compact for $2 < p < 2^\ast_s$ 
due to Lions \cite{L-82}, one has 
	\begin{equation}\label{eq:3.15}
		\begin{aligned}
			\limsup_{n \to \infty} \int_{ [ |x| \geq R_{\eta,p_0} ] \cap [ 2t_\eta \leq |u_n| \leq T_\eta ] } 
			|f_+(u_n) u_n| \rd x 
			&\leq \limsup_{n \to \infty} \int_{ [ |x| \geq R_{\eta,p_0} ] } 
			|f_+ (v_n) v_n | \rd x 
			\\
			&\leq C_{\eta,p_0} \limsup_{n \to \infty} \| v_n \|_{L^{p_0}([ |x| \geq R_{\eta,p_0} ])}^{p_0}
			\\
			&= C_{\eta,p_0} \| v_0 \|_{L^{p_0}([ |x| \geq R_{\eta,p_0} ])}^{p_0} \leq \eta. 
		\end{aligned}
	\end{equation}

From \eqref{eq:3.13} through \eqref{eq:3.15}, we observe that 
	\[
		\limsup_{n \to \infty} \int_{|x| \geq R_{\eta,p_0}} 
		\left| f_+(u_n) u_n - f_+(u_0) u_0 \right| \rd x \leq C_1 \eta
	\]
for some $C_1>0$, which is independent of $\eta$. 
Hence, from \eqref{eq:3.12} we infer that 
	\[
		\begin{aligned}
			\limsup_{n \to \infty} \left| \int_{\RN} f_+(u_n) u_n - f_+(u_0) u_0 \rd x \right| 
			&\leq \limsup_{n \to \infty}\left( \int_{|x| \leq R_{\eta,p_0}} 
			+ \int_{|x| \geq R_{\eta,p_0}} \right) 
			\left| f_+(u_n) u_n - f_+(u_0) u_0  \right| \rd x 
			\\
			&\leq C_2 \eta
		\end{aligned}
	\]
for some $C_2$ which does not depend on $\eta$. 
Since $\eta > 0 $ is arbitrary, \eqref{eq:3.11} holds.

	Now, from $f_-(t) t \geq 0$ for each $ t \in \R$, 
Fatou's lemma, $\| \rd I_\e(u_n) \|_{\cDs{\RN}^\ast} \to 0$, \eqref{eq:3.10} and \eqref{eq:3.11}, 
it follows that
	\[
		\begin{aligned}
			\| u_0 \|_{\cDs{\RN}}^2 &\leq \liminf_{n \to \infty}  \| u_n \|_{\cDs{\RN}}^2 
			\\
			&\leq \limsup_{n \to \infty} \| u_n \|_{\cDs{\RN}}^2 
			\\
			&= \limsup_{n \to \infty} \left\{ 
			\rd I_\e (u_n) u_n + \int_{\RN} f_+(u_n) u_n \rd x - \int_{\RN} f_{\e,-}(u_n) u_n \rd x
			 \right\}
			 \\
			 &\leq \int_{\RN} f_+(u_0) u_0 \rd x - \int_{\RN} f_{\e,-} (u_0) u_0 \rd x 
			 = \| u_0 \|_{\cDs{\RN}}^2.
		\end{aligned}
	\]
Hence, $\| u_n \|_{\cDs{\RN}}^2 \to \| u_0 \|_{\cDs{\RN}}^2$ and 
$\| u_n - u_0 \|_{\cDs{\RN}} \to 0$ hold as $n \to \infty$. 
This completes the proof. 
	\end{proof}

	\begin{remark}\label{remark:3.5}
From the above argument, under \eqref{eq:3.6}, we may observe that the functionals 
	\[
		u \mapsto \int_{\RN} F_+(u) \rd x, \quad u \mapsto \int_{\RN} f_+(u)u \rd x : 
		\cDsr \to \R
	\]
are weakly continuous. 
	\end{remark}

Now we verify Theorem \ref{theorem:1.1}:

	\begin{proof}[Proof of Theorem \ref{theorem:1.1}]
By Theorem \ref{theorem:2.4} and Lemmas \ref{lemma:3.2} and \ref{lemma:3.4}, there exist 
$(u_{\e,k})_{0<\e \leq 1, 1 \leq k}$ such that 
	\[
		\begin{aligned}
			&I_\e( u_{\e,k}  ) = c_{\e,k} \to \infty, \quad 
			\rd I_\e ( u_{\e,k} ) = \rd_u J_\e(0,u_{\e,k}) = 0, 
			\\
			&
			P_\e (u_{\e,k}) := \rd_\theta J_\e (0, u_{\e,k}) 
			= \frac{N-2s}{2} \| u_{\e,k} \|_{\cDs{\RN}}^2 
			- N \int_{\RN} F_+(u_{\e,k}) - F_{\e,-} (u_{\e,k}) \rd x . 
		\end{aligned}
	\]
By Remark \ref{remark:3.3} and 
	\begin{equation}\label{eq:3.16}
		\frac{s}{N} \| u_{\e,k} \|_{\cDs{\RN}}^2 
		= I_\e(u_{\e,k}) - \frac{P_\e(u_{\e,k})}{N} = c_{\e,k} \leq \ov{c}_k < \infty 
		\quad {\rm for\ each}\ \e \in (0,1],
	\end{equation}
one sees that 
$(u_{\e,k})_{0<\e \leq 1}$ is bounded in $\cDsr$ and 
up to a subsequence, we may assume $u_{\e,k} \rightharpoonup u_{0,k}$ weakly in $\cDs{\RN}$.

	Let $\varphi \in C^\infty_0(\RN)$. By the weak convergence in $\cDsr$, $\rd I_\e (u_{\e,k}) = 0$, 
\eqref{align:3.5} and Strauss' lemma, one has  
	\begin{equation}\label{eq:3.17}
		\la u_{0,k} , \varphi \ra_{\cDs{\RN} } 
		= \lim_{\e \to 0} \la u_{\e,k} , \varphi \ra_{\cDs{\RN}} 
		= \lim_{\e \to 0} \int_{\RN} \left\{ f_+(u_{\e,k}) - f_{\e,-} (u_{\e,k})  \right\} \varphi \rd x 
		= \int_{\RN} f(u_{0,k}) \varphi \rd x. 
	\end{equation}
Thus, $u_{0,k}$ is a solution of \eqref{eq:1.1}.

	Next, from Remark \ref{remark:3.5}, it follows that 
	\begin{equation}\label{eq:3.18}
		\int_{\RN} F_+(u_{\e,k}) \rd x \to \int_{\RN} F_+(u_{0,k}) \rd x, \quad 
		\int_{\RN} f_+(u_{\e,k}) u_{\e,k} \rd x \to \int_{\RN} f_+(u_{0,k}) u_{0,k} \rd x.
	\end{equation}
By $\rd I_\e(u_{\e,k}) u_{\e,k} = 0$, $I_\e (u_{\e,k}) = c_{\e,k}$ and \eqref{eq:3.16}, we observe that 
	\begin{equation}\label{eq:3.19}
		\sup_{ 0 < \e \leq 1} \int_{\RN} F_{\e,-}(u_{\e,k}) \rd x 
		+ \sup_{ 0 < \e \leq 1} \int_{\RN} f_{\e,-}(u_{\e,k}) u_{\e,k} \rd x < \infty.
	\end{equation}
Hence, Fatou's lemma with \eqref{eq:3.19} gives 
	\begin{equation}\label{eq:3.20}
		\begin{aligned}
			\infty &> \liminf_{\e \to 0} \int_{\RN} F_{\e,-} (u_{\e,k}) \rd x \geq \int_{\RN} F_{-} (u_{0,k}) \rd x, 
			\\
			\infty &> \liminf_{\e \to 0} \int_{\RN} f_{\e,-} (u_{\e,k}) u_{\e,k} \rd x 
			\geq \int_{\RN} f_{-} (u_{0,k}) u_{0,k} \rd x .
		\end{aligned}
	\end{equation}
Thus, $| I_0(u_{0,k}) | < \infty$.

	Next, as in the proof of Lemma \ref{lemma:3.1}, \eqref{eq:3.17} implies 
	\[
		\la u_{0,k} , \varphi_R u_{0,k} \ra_{\cDs{\RN} } = \int_{\RN} f_+(u_{0,k}) u_{0,k} \varphi_R \rd x 
		- \int_{\RN} f_- (u_{0,k}) u_{0,k} \varphi_R \rd x
	\]
where $\varphi_R \in C^\infty_0(\RN)$ where $0 \leq \varphi_R \leq 1$, 
$\varphi_R \equiv 1$ on $B_R(0)$ and $\varphi_R= 0$ on $\RN \setminus B_{2R} (0)$. 
By \eqref{eq:3.20}, $f_\pm(t) t \geq 0$ for $t \in \R$ 
and the dominated convergence theorem, letting $R \to \infty$, we obtain 
	\begin{equation}\label{eq:3.21}
		\| u_{0,k} \|_{\cDs{\RN}}^2 = \int_{\RN} f_+(u_{0,k}) u_{0,k} \rd x 
		- \int_{\RN} f_-(u_{0,k}) u_{0,k} \rd x 
		= \int_{\RN} f(u_{0,k}) u_{0,k} \rd x. 
	\end{equation}

	Now \eqref{eq:3.18}, \eqref{eq:3.20} and \eqref{eq:3.21} yield 
	\[
		\begin{aligned}
			\| u_{0,k} \|_{\cDs{\RN}}^2 
			& \leq \liminf_{\e \to 0} \| u_{\e,k} \|_{\cDs{\RN}}^2 
			\\
			&\leq \limsup_{\e \to 0} \| u_{\e,k} \|_{\cDs{\RN}}^2 
			\\
			&= \limsup_{\e \to 0} 
			\left[ \rd I_\e (u_{\e,k}) u_{\e,k} + \int_{\RN} f_+(u_{\e,k}) u_{\e,k} 
			- \int_{\RN} f_{\e,-} (u_{\e,k}) u_{\e,k} \rd x \right]
			\\
			&\leq \int_{\RN} f_+(u_{0,k}) u_{0,k} \rd x - \int_{\RN} f_-(u_{0,k}) u_{0,k} \rd x 
			= \| u_{0,k} \|_{\cDs{\RN}}^2. 
		\end{aligned}
	\]
Therefore, we have $\| u_{\e,k} \|_{\cDs{\RN}} \to \| u_{0,k} \|_{\cDs{\RN}}$ and 
$\| u_{\e,k} - u_{0,k} \|_{\cDs{\RN}} \to 0$. 
From $c_{1,k} \to \infty$ and 
$c_{1,k} \leq c_{\e,k} = s \| u_{\e,k} \|_{D^s}^2 / N $ due to 
\eqref{eq:3.16}, we conclude that $c_{1,k} \leq s \| u_{0,k} \|_{\cDs{\RN} }^2 / N $ and 
$\| u_{0,k} \|_{\cDs{\RN}} \to \infty$. 
Thus, we complete the proof. 
\end{proof}

	\begin{remark}\label{remark:3.6}
(i) Regarding $I_{0} (u_{0,k}) \to \infty$ as $k \to \infty$, 
we may prove this claim provided either every solution of \eqref{eq:1.1} 
satisfies the Pohozaev identity or the conditions (f0), (f1), (f2') and (f3) with 
	\[
		- \infty < \liminf_{|t| \to 0} \frac{f(t)}{|t|^{2^\ast_s-2}t }.
	\]
Notice that in the latter case, $f_\e (t) = f$ for sufficiently small $\e>0$. 
Thus, in either case, we may prove $P_0(u_{0,k}) = 0$ in the above proof 
by $u_{\e,k} \to u_{0,k}$ strongly in $\cDs{\RN}$. 
This yields  $I_0(u_{0,k}) = I_0(u_{0,k}) - P_0(u_{0,k}) = s \| u_{0,k} \|_{\cDs{\RN}}^2 / N \to \infty$.

(ii) 
When $N \geq 3$ and $s = 1$, by changing 
$\cDs{\RN}$ and $\| u \|_{\cDs{\RN}}$ to $\cD^1(\RN)$ and $\| \nabla u \|_{L^2(\RN)}$, 
it is not difficult to see that all the arguments in the above work. 
Moreover, according to Berestycki and Lions \cite[Proposition 1]{BL-83-1}, 
the Pohozaev identity is satisfied for each solution $u$ of \eqref{eq:1.1} 
with $\int_{\RN} F(u) \rd x < \infty$. Hence, the solutions $u_{0,k}$ found in the above 
for the case $N \geq 3$ and $s=1$ enjoy $P_0(u_{0,k}) = 0$. 
Thus, we also get $I_0(u_{0,k}) \to \infty$ by (i) 
and we may provide another proof for 
the results of \cite{BL-83-3,Stru-82,Stru-86}. 
	\end{remark}

	Next, we prove Theorem \ref{theorem:1.3} via Theorem \ref{theorem:2.3}:

	\begin{proof}[Proof of Theorem \ref{theorem:1.3}]
We first remark that under the assumptions of Theorem \ref{theorem:1.3}, 
$I_0 \in C^1(\cDs{\RN}, \R)$, and 
$X:= \cDsr$ and $I_0$ satisfy \eqref{T1} and \eqref{T2}. 
Furthermore, we may observe that the proofs of Lemmas \ref{lemma:3.2} and \ref{lemma:3.4} 
still work for $(X,I_0)$ and 
$I_0$ satisfies the assumptions of Theorem \ref{theorem:2.3}. 
Hence, by Theorem \ref{theorem:2.3}, we may find a $u_0 \in \cDsr$ satisfying 
	\[
		\rd J(0,u_0) = 0, \quad 0< I(u_0) = c_{mp} 
		= \inf_{ \gamma \in \Gamma} \max_{0 \leq t \leq 1} I_0(\gamma(t))
	\]
where $\Gamma := \set{ \gamma \in  C([0,1] ,\cDsr ) | \gamma(0) = 0, \ I(\gamma(1)) < 0 }$. 
Since $I_0(u) \geq 0$ for $\| u \|_{\cDs{\RN}} \leq \rho_0$, we may omit the condition 
$\rho_0 < \| \gamma (1) \|_{\cDs{\RN}}$ from the definition of $\Gamma$. 
We also note that $P_0(u_0) = \rd_{\theta} J(0,u_0) = 0$.

	Finally, we shall show the last equality in \eqref{eq:1.4}. 
To this end, let $v \in \cDs{\RN}$ be a solution of \eqref{eq:1.1} with 
$P_0(v) = 0$. Our aim is to prove $c_{mp} \leq I_0(v)$. 
Denote by $v^\ast$ the Schwarz rearrangement of $v$. 
Then the following hold (see \cite{AL-89,LL-01}):
	\[
		\| v^\ast \|_{\cDs{\RN}} \leq \| v \|_{\cDs{\RN}}, \quad 
		\int_{  \RN } F(v^\ast ) \rd x = \int_{  \RN } F(v) \rd x.
	\]
From this, it follows that
	\[
		I_0 ( v^\ast ( \theta^{-1} \cdot ) ) 
		= \frac{\theta^{ N-2s }}{2} \| v^\ast \|_{\cDs{\RN}}^2 
		- \theta^{N} \int_{  \RN } F(v^\ast) \rd x 
		\leq 
		I_0( v ( \theta^{-1} \cdot ) ) 
		\quad \text{for all $\theta > 0$}.
	\]

	On the other hand, by $P_0(v) = 0$, we observe that 
	\[
		\int_{  \RN } F(v) \rd x > 0, \quad I_0( v ( \theta^{-1} \cdot ) ) \to - \infty 
		\ \text{as $\theta \to \infty$}, \quad 
		\max_{0 \leq \theta } I_0( v ( \theta^{-1} \cdot)  ) 
		= I_0(v).
	\]
By these facts, a path defined by 
$\gamma_v(0) := 0$, $\gamma_{v} (t) := v^\ast ( (t\theta_0)^{-1} \cdot )$ for sufficiently large $\theta_0>0$ 
satisfies $\gamma_{v} \in \Gamma$. Thus, 
	\[
		I_0(u_0) = c_{mp} \leq \max_{0 \leq t \leq 1} I_0( \gamma_v (t) ) \leq I_0(v)
	\]
and we complete the proof. 
	\end{proof}


\section{Proof of Theorem \ref{theorem:1.5}}\label{section:4}


	In this section, we prove Theorem \ref{theorem:1.5} via 
Theorems \ref{theorem:2.3} and \ref{theorem:2.4}. 
We recall the notation in Introduction: 
	\[
		\cO_1 := O(m_1) \times O(m_1) \times \{ \mathrm{id}_{\R^{N-2m_1}} \} \subset O(N), \quad 
		\cO_2 := O(m_2) \times O(m_2) \times O(N-2m_2) \subset O(N)
	\]
and 
	\[
		\begin{aligned}
		\cD_{\cO_1}^s &:= 
		\Set{ u \in \cDs{\RN} |
		u(x_1,x_2,x_3) = u(|x_1|,|x_2|,x_3), \ 
		u(x_2,x_1,x_3) = - u(x_1,x_2,x_3)  },
		\\ 
		\cD_{\cO_2}^s &:= 
			\Set{ u \in \cDs{\RN} |
			u(x_1,x_2,x_3) = u(|x_1|,|x_2|,|x_3|), \ 
			u(x_2,x_1,x_3) = - u(x_1,x_2,x_3)  },
		\\
		H^s_{\cO_1} &:= 
		\Set{ u \in H^s(\RN) | 
			u (x_1,x_2,x_3) = u(|x_1|,|x_2|,x_3), \ 
			u(x_2,x_1,x_3) = - u(x_1,x_2,x_3)},
		\\
		H^s_{\cO_2} &:= 
		\Set{ u \in H^s(\RN) | 
		u (x_1,x_2,x_3) = u(|x_1|,|x_2|,|x_3|), \ 
		u(x_2,x_1,x_3) = - u(x_1,x_2,x_3)}.
		 \end{aligned}
	\]
We shall find solutions in $\cD^s_{\cO_i}$ and $H^s_{\cO_i}$ which 
are nonradial and sign-changing. 
Furthermore, remark that $T_\theta u (x) := u(e^{-\theta} x)$ satisfies \eqref{T1} 
with $\cD^s_{\cO_i}$ and $H^s_{\cO_i}$.

	Before proving Theorem \ref{theorem:1.5}, we prepare one lemma and 
it is proved in \cite[Lemma 2.1]{M-19} when $s=1$.

	\begin{lemma}\label{lemma:4.1}
		Let $(u_n) \subset \cDs{\RN}$ be bounded and satisfy 
		\begin{equation}\label{eq:4.1}
			\lim_{n \to \infty} \sup_{z \in \Z^N} \| u_n \|_{L^2(z+Q)} = 0, \quad 
			Q := \left[-\frac{1}{2},\frac{1}{2}\right]^N.
		\end{equation}
	Then $\int_{  \RN } G(u_n) \rd x \to 0$ as $n \to \infty$ for every $G \in C(\R)$ 
	satisfying 
		\begin{equation}\label{eq:4.2}
			\lim_{t \to 0} \frac{G(t)}{|t|^{2^\ast_s}} = 0 
			= \lim_{|t| \to \infty} \frac{G(t)}{|t|^{2^\ast_s}}.
		\end{equation}
	\end{lemma}

	\begin{proof}
Assume that $(u_n)$ is bounded and satisfies \eqref{eq:4.1}. 
We argue in a similar way to the proof of Lemma \ref{lemma:3.4}. 
Set 
	\[
		C_0 := \sup_{n \geq 1} \| u_n \|_{L^{2^\ast_s}(\RN)}  < \infty.
	\]
By \eqref{eq:4.2}, for any $\eta > 0$, we may choose $0<2t_\eta < T_\eta$ so that 
	\begin{equation}\label{eq:4.3}
		|G(t) | \leq \eta |t|^{2^\ast_s} \ \text{for $|t| \leq 2t_\eta$}, \quad 
		|G(t) | \leq \eta |t|^{2^\ast_s} \ \text{for $|t| \geq T_\eta$}.
	\end{equation}
Pick up $\chi_\eta \in C^\infty_0(\R)$ satisfying 
	\[
		\chi_\eta(-t) = - \chi_\eta(t), \quad 
		\left| \chi_\eta (t) \right| \leq |t|, \quad 
		\chi_\eta (t) = 0 \quad \text{if $|t| \leq t_\eta$ or $|t| \geq 2T_\eta$}, \quad 
		\chi_\eta (t) = t \quad \text{if $2 t_\eta \leq |t| \leq T_\eta$}.
	\]
Finally, set $v_n(x) := \chi_\eta (u_n(x))$. As in the proof of Lemma \ref{lemma:3.4}, 
we observe that $(v_n)$ is bounded in $H^s(\RN)$. In addition, by $|\chi_\eta(t)| \leq |t|$, 
$(v_n)$ also satisfies \eqref{eq:4.1}. Hence, by Felmer, Quaas and Tan \cite[Lemma 2.2]{FQT-12} (or \cite[Lemma 4.5]{I-17}), 
for any $ p \in (2,2^\ast_s)$, $\| v_n \|_{L^p(\RN)} \to 0$ as $n \to \infty$. Fix $p \in (2,2^\ast_s)$ and 
choose a $C_{p,\eta}$ so that $|G(t)| \leq C_{p,\eta} |t|^p$ for every $t_\eta \leq |t| \leq 2T_\eta$. 
Since 
	\[
		\int_{[ 2t_\eta \leq |u_n| \leq T_\eta ]} \left| G(u_n(x)) \right| \rd x 
		\leq \int_{ [ 2t_\eta \leq |v_n| \leq T_\eta ] } \left| G(v_n(x)) \right| \rd x 
		\leq C_{p,\eta} \| v_n \|_{L^p(\RN)}^p \to 0,
	\]
it follows from \eqref{eq:4.3} that 
	\[
		\limsup_{n \to \infty} \left| \int_{  \RN } G(u_n) \rd x \right| 
		\leq \limsup_{n \to \infty} 
		\int_{  [|u_n| \leq 2t_\eta] \cup [|u_n| \geq T_\eta]} |G(u_n(x))| \rd x 
		\leq 2 C_0^{2^\ast_s} \eta.
	\]
Noting that $\eta > 0$ is arbitrary, we have $\int_{  \RN } G(u_n) \rd x \to 0$. 
	\end{proof}

	We first prove Theorem \ref{theorem:1.5} (i) and (ii).

	\begin{proof}[Proof of Theorem \ref{theorem:1.5} (i) and (ii)]
(i) Suppose (f0), \eqref{eq:1.3} and (f3). Then $I_0 \in C^1( \cD^s_{\cO_1} , \R )$. 
Moreover, as in Lemma \ref{lemma:3.2}, there exists a $\rho_0>0$ such that 
	\[
		\inf_{ \| u \|_{\cDs{\RN}} \leq \rho_0} I_0(u) = 0 < \inf_{ \| u \|_{\cDs{\RN}} = \rho_0} I_0(u).
	\]

	Next, we shall show the existence of $w_0 \in \cD_{\cO_1}^s$ satisfying $I_0(w_0) < 0$. 
To this end, we borrow an idea from \cite[Lemmas 4.2 and 4.3]{JL-18}. 
For $R>0$, set 
	\[
		\omega_R(t) := \zeta_1 \ \text{if $|t| \leq R$}, \quad 
		\omega_R(t) := \zeta_1(R+1-|t|) \ \text{if $R < |t| \leq R+1$}, \quad 
		\omega_R(t) := 0 \ \text{if $|t| > R+1$}.
	\]
Fix a $\varphi_R \in C^1_0(\R)$ with $0 \leq \varphi_R(t) \leq 1$, 
$\varphi_R(t) \equiv 1$ if $|t| \leq R^2$ and $\varphi_R(t) = 0$ 
if $|t| \geq R^2 + 1$. We define 
	\[
		w_R(x_1,x_2,x_3) := \varphi_{R} (|x_1|) \varphi_{R}(|x_2|) \varphi_R (|x_3|) 
		\left\{ \omega_R(|x_1|) - \omega_R ( |x_2| ) \right\} \in H^1_{\cO_1} 
		\subset \cD_{\cO_1}^s.
	\]
Due to $F(-t) = F(t)$ and the (anti)symmetry for functions in $\cD_{\cO_1}^s$, we set 
	\[
		\begin{aligned}
			I_{R,1} &:= \Set{ (x_1,x_2,x_3) \in \RN |\  |x_1| \leq R, \ |x_2| \leq R+1, \ |x_3| \leq R^2+1},
			\\
			I_{R,2} &:= \Set{ (x_1,x_2,x_3) \in \RN |\  |x_1| \leq R, \ R + 1 \leq |x_2| \leq R^2, \ |x_3| \leq R^2 },
			\\
			I_{R,3} &:= \Set{ (x_1,x_2,x_3) \in \RN |\  |x_1| \leq R, \ R + 1 \leq |x_2| \leq R^2, \ 
				R^2 \leq |x_3| \leq R^2 +1 },
			\\
			I_{R,4} &:= \Set{ (x_1,x_2,x_3) \in \RN |\  |x_1| \leq R, \ R^2 \leq |x_2| \leq R^2 + 1, \ 
				 |x_3| \leq R^2 +1 },
			\\
			I_{R,5} &:= \Set{ (x_1,x_2,x_3) \in \RN |\  R \leq |x_1| \leq R+1, \  |x_2| \leq R^2+1, 
				\ |x_3| \leq R^2+1 }
		\end{aligned}
	\]
when $N-2m_1 \geq 1$. When $N-2m_1=0$, we ignore the third component in the above and 
define $I_{R,1}, I_{R,2}, I_{R,4}, I_{R,5}$ since $I_{R,2}$ and $I_{R,3}$ are same in this case.

	For large $R>1$,  it is easily seen from 
$ (R^\alpha + 1)^\beta - R^{\alpha \beta} = O( R^{\alpha(\beta-1)} ) $ that 
	\begin{equation}\label{eq:4.4}
		\begin{aligned}
			\left| \int_{I_{R,1} \cup I_{R,3} \cup I_{R,4} \cup I_{R,5}  } F(w_R) \rd x \right|
			&\leq C \left( R^{2N-2m_1} + R^{2N-m_1-2} + R^{2N-m_1-1} \right) 
			& &\text{if $N-2m_1 \geq 1$},
				\\
			\left| \int_{ I_{R,1} \cup I_{R,4} \cup I_{R,5} } F(w_R) \rd x \right|
			&\leq C \left( R^{2m_1} + R^{3m_1-2} + R^{3m_1-1} \right) & &\text{if $N-2m_1=0$}.
			\end{aligned}
	\end{equation}
On the other hand, since $F(w_R(x_1,x_2,x_3)) = F(\zeta_1) > 0$ for $x \in I_{R,2}$, we have 
	\begin{equation}\label{eq:4.5}
			\int_{ I_{R,2} } F(w_R) \rd x 
			\geq 
			\left\{\begin{aligned}
			& c_{N,m_1} R^{2N-m_1} 
			& &\text{if $N-2m_1 \geq 1$},
			\\
			& c_{N,m_1} R^{3m_1} 
			& &\text{if $N-2m_1 =0$},
		\end{aligned}\right.
	\end{equation}
for some $c_{N,m_1} > 0$.

	By $m_1 \geq 2$, \eqref{eq:4.4} and \eqref{eq:4.5}, 
for sufficiently large $R_0>0$, we obtain 
	\[
		\int_{  \RN } F( w_{R_0} ) \rd x \geq 1.
	\]
Now for sufficiently large $\theta_0>0$, it follows that 
	\[
		I_0(w_{R_0} (e^{-\theta_0} \cdot ) ) 
		\leq \frac{e^{(N-2s)\theta_0}}{2} \| w_{R_0} \|_{\cDs{\RN}}^2 
		- e^{N\theta_0} < 0.
	\]
Thus, $w_0(x) := w_{R_0} (e^{-\theta_0} x) \in \cD_{\cO_1}$ satisfies $I_0(w_0) < 0$.

	Applying Theorem \ref{theorem:2.3} for $(\cD_{\cO_1}, I_0)$ without \PSPc, 
we may find $(u_n) \subset \cD_{\cO_1}^s$ such that 
	\[
		I_0(u_n) = \frac{1}{2} \| u_n \|_{\cDs{\RN}}^2 - \int_{  \RN } F(u_n) \rd x \to c_{mp} > 0, \quad 
		\left\| \rd J(0,u_n) \right\|_{(\R \times \cD_{\cO_1}^s)^\ast} \to 0.
	\]
Moreover, as in Lemma \ref{lemma:3.4}, we can prove that $(u_n)$ is bounded in $\cDs{\RN}$. If 
	\[
		\sup_{ z \in \Z^N} \| u_n \|_{L^2(z+Q)} \to 0,
	\]
then applying Lemma \ref{lemma:4.1} for $f_+(t)t$, we infer from $\rd_u J(0,u_n) u_n \to 0$ that 
	\[
		\| u_n \|_{\cD^s(\RN)}^2 + \int_{  \RN } f_-(u_n) u_n \rd x = o(1) + \int_{  \RN } f_+(u_n)u_n \rd x 
		= o(1),
	\]
which implies $\| u_n \|_{\cDs{\RN}} \to 0$. 
Since it follows from \eqref{eq:1.3} that $|F_-(t)| \leq C_0|t|^{2^\ast_s}$ for every $t \in \R$, 
we have a contradiction:
	\[
		0< c_{mp} = \lim_{n \to \infty} I_0(u_n) = 0.
	\]

	Now choose $(z_n)_{n=1}^\infty \subset \Z^N$ so that 
$\| u_n \|_{L^2(z_n+Q)} \to c >0$. 
Since $m_1 \geq 2$ and $(u_n)_n \subset \cDs{\RN}$ is bounded, 
replacing $Q$ by $R_0Q$ for sufficiently large $R_0>1$, 
we may suppose $z_n = (0,0,x_{n,3})$  (cf. Willem \cite[Theorem 1.24]{W}).  
Let $v_n(x) := u_n(x+z_n) \in \cD_{\cO_1}^s $ and 
$v_n \rightharpoonup v_0 \not\equiv 0$ weakly in $\cD_{\cO_1}^s$. 
From $\| v_n \|_{L^2(Q)} \to c > 0$ and $\| \rd J(0,u_n) \|_{(\R \times \cD^s_{\cO_1})^\ast} \to 0$, 
we deduce that $v_0 \not\equiv 0$ and $I_0'(v_0) = 0$ 
due to \eqref{eq:1.3}, Strauss' lemma and 
the principle of symmetric criticality due to Palais \cite[Theorem 1.28]{W}. 
Thus, $v_0$ is a nontrivial solution of \eqref{eq:1.1} and the statement (i) holds.

(ii) 
As in the proof of Theorem \ref{theorem:1.1}, we may assume that 
$f$ satisfies (f0), (f1), (f2') and (f3). Moreover, let $f_{\e,-}$ and $f_\e$ be as in section \ref{section:3}. 
Notice that by the principle of symmetric criticality due to Palais \cite[Theorem 1.28]{W}, 
we shall find critical points of $I_\e$ in $\cD_{\cO_2}^s$ by applying Theorem \ref{theorem:2.4} for 
$\cD^s_{\cO_2}, I_\e, T_\theta u(x) = u( e^{-\theta} x )$. 
It is easily seen that \eqref{T1} and \eqref{T2} are satisfied. 
As in the proof of Lemma \ref{lemma:3.2}, we may show that the assumption (i) in Theorem \ref{theorem:2.4} holds. 
For (ii), in \cite[Lemma 4.2]{JL-18}, under (f0) and (f3), the following maps $\ov{\gamma}_k$ are constructed: 
for each $k \geq 1$ and $\sigma \in S^{k-1}$, 
	\begin{equation}\label{eq:4.6}
		\ov{\gamma}_k \in C( S^{k-1} , H^1_{\cO_2} ), \ 
		\ov{\gamma}_k(-\sigma) = - \ov{\gamma}_k (\sigma), \ 
		{\rm supp}\, (\ov{\gamma}_k (\sigma)) \subset B_{R_k}(0), 
		\int_{\RN} F ( \ov{\gamma} (\sigma) ) \rd x \geq  1.
	\end{equation}
Now, we observe that 
for sufficiently large $\theta_k$, $\gamma_k(\sigma) (x) := \ov{\gamma}_k(\sigma) (e^{-\theta_k}x)$ satisfies (ii).

	Next, exploiting the argument in \cite[Theorem 1.24]{W}, 
the embedding $H^s_{\cO_2} \subset L^p(\RN)$ is compact for $2 < p < 2^\ast_s$. 
Hence, we may verify that $I_\e$ satisfies \PSPc in $\cD_{\cO_2}^s$ as in Lemma \ref{lemma:3.4}.

	From Theorem \ref{theorem:2.4}, there exist $(u_{\e,k})_{0<\e \leq 1, 1 \leq k}$ such that 
	\[
		I_\e(u_{\e,k}) = c_{\e,k}, \quad \rd I_\e (u_{\e,k}) = 0, \quad 
		P_\e (u_{\e,k}) = 0, \quad 
		c_{1,k} \leq c_{\e,k} \leq \ov{c}_{k} := \max_{\sigma \in D^k} 
		I_0 \left( \gamma_{0,k} (\sigma) \right) < \infty
	\]
where $c_{1,k} \to \infty$ and $\gamma_{0,k}(\sigma) := |\sigma| \ov{\gamma}_{k} (\sigma/|\sigma|)$ if $|\sigma|>0$ 
and $\gamma_{0,k} (0) = 0$. 
Now the rest of the proof is identical to that of Theorem \ref{theorem:1.1} and 
we may show 
$u_{\e} \to u_{0,k}$ strongly in $\cDs{\RN}$ where $u_{0,k}$ is a solution of \eqref{eq:1.1}. 
Thus, Theorem \ref{theorem:1.5} (ii) holds. 
	\end{proof}

	Next, we treat Theorem \ref{theorem:1.5} (iii) and (iv). 
First we remark that when $a=0$ we may assume (f2') without loss of generality 
due to Lemma \ref{lemma:3.1}. 
When $a>0$, we prove

	\begin{lemma}\label{lemma:4.2}
		Assume $f$ satisfies \textup{(F0)--(F3)} with $a>0$ and there exists a $\zeta_2 > \zeta_1$ such that
		$f(t) = 0$ for all $ t \geq \zeta_2$. Let $u \in H^s(\RN)$ be any solution of \eqref{eq:1.2}. 
		Then $\| u \|_{L^\infty(\RN)} \leq \zeta_2$. 
	\end{lemma}

	\begin{proof}
From Fall and Felli \cite[Proposition 6]{FF-15}, it follows that 
	\[
		\begin{aligned}
			\la u , v \ra_{s,a} =  a^s \int_{\RN} uv \rd x + 
			C_{N,s} \int_{\R^{2N}} \frac{( u(x) - u(y) ) (v(x) - v(y) ) }{|x-y|^{\frac{N+2s}{2}} }
			K_{ \frac{N+2s}{2} } \left( \sqrt{a} |x-y| \right) \rd x \rd y
		\end{aligned}
	\]
where $C_{N,s} > 0$ and 
$K_\nu$ stands for the modified Bessel function of the second kind with the order $\nu$. 
As in Lemma \ref{lemma:3.1}, putting 
$v_0(x) := \max\{ \zeta_2 , u(x) \} - \zeta_2 =(u(x) - \zeta_2 )_+ \in H^s(\RN)$, 
we obtain 
	\[
		\begin{aligned}
			0 &= \int_{\RN} f(u) v_0 \rd x =  \la u ,  v_0 \ra_{s,a} 
		\end{aligned}
	\]
Noting $K_\nu (z) > 0$ for $z > 0$, as in Lemma \ref{lemma:3.1}, we have 
	\[
		\begin{aligned}
			0 = \la u , v_0 \ra_{s,a} 
			&= a^s \int_{\RN} u v_0 \rd x + C_{N,s} \int \frac{ (u(x) - u(y))  (v_0(x) - v_0(y ) ) }{|x-y|^{\frac{N+2s}{2}}} 
			K_{ \frac{N+2s}{2} } \left( \sqrt{a} |x-y| \right) \rd x \rd y
			\\
			& \geq a^s \int_{\RN} v_0^2 \rd x 
			+ C_{N,s} \int_{\R^{2N}}\frac{|v_0(x) -  v_0(y) |^2 }{|x-y|^{\frac{N+2s}{2}} }
			K_{ \frac{N+2s}{2} } \left( \sqrt{a} |x-y| \right) \rd x \rd y
			= \| v_0 \|_{s,a}^2.
		\end{aligned}
	\]
Hence, $v_0 \equiv 0$ in $\RN$ and $u \leq \zeta_2$ holds. 
Similarly we can prove $- \zeta_2 \leq u$ and Lemma \ref{lemma:4.2} holds. 
	\end{proof}

	From Lemma \ref{lemma:4.2}, for the case $a>0$, we may also assume (f2') without loss of generality 
as in section \ref{section:3}. 
Under (F0), (F1), (f2') and (F3), we define 
	\[
		I_a(u) := \frac{1}{2} \| u \|_{s,a}^2 - \int_{\RN} F(u) \rd x 
		\in C^1( H^s(\RN) , \R ).
	\]
Remark also that $T_\theta u(x) := u(e^{-\theta} x)$ and $J(\theta,u) :=I_a(T_\theta u)$ 
enjoy \eqref{T1} and \eqref{T2}. 
Next, we shall verify

	\begin{lemma}\label{lemma:4.3}
		\begin{enumerate}
			\item 
				For each $a \geq 0$, there exists a $w_{0,a} \in H^s_{\cO_1}$ such that 
				$I_a (w_{0,a}) < 0$. 
			\item
				For each $a \geq 0$, $I_a$ in $H^s_{\cO_2}$ satisfies the conditions \textup{(i)} and \textup{(ii)} in 
				Theorem \ref{theorem:2.4}. 
			\item
				For each $a \geq 0$ and $c \in \R$, $I_a$ in $H^s_{\cO_2}$ satisfies \PSPc. 
		\end{enumerate}
	\end{lemma}

	\begin{proof}
(i) Write $G(t) := F(t) - a^s t^2/2$. By (F0) and (F3), as in the proof of Theorem \ref{theorem:1.5} (i), 
we may find $w_{R_0} \in H^s_{\cO_1}$ such that 
	\[
		\int_{  \RN } G( w_{R_0} ) \rd x \geq 1. 
	\]
Next, notice that 
	\[
		\begin{aligned}
			I_a ( w_{R_0} ( \theta^{-1} \cdot )  ) 
			&= \frac{1}{2} \| w_{R_0} ( \theta^{-1} \cdot ) \|_{s,a}^2 - \int_{  \RN } F(w_{R_0} ( \theta^{-1} x ) ) \rd x 
			\\
			&= \frac{1}{2} \left[ \| w_{R_0} ( \theta^{-1} \cdot ) \|_{s,a}^2 
			- a^s \| w_{R_0} ( \theta^{-1} \cdot ) \|_{L^2(\RN)}^2  \right] 
			- \theta^N \int_{  \RN } G(w_{R_0}) \rd x
			\\
			&\leq \frac{1}{2} \left[ \| w_{R_0} ( \theta^{-1} \cdot ) \|_{s,a}^2 
			- a^s \| w_{R_0} ( \theta^{-1} \cdot ) \|_{L^2(\RN)}^2  \right] 
			- \theta^N 
		\end{aligned}
	\]
and that 
	\[
		\begin{aligned}
			\| w_{R_0} ( \theta^{-1} \cdot ) \|_{s,a}^2 
			- a^s \| w_{R_0} ( \theta^{-1} \cdot ) \|_{L^2(\RN)}^2 
			&= \theta^{2N} \int_{  \RN } \left\{ \left( 4\pi^2 |\xi|^2 + a \right)^s - a^s \right\} 
			\left| \scrF w_{R_0} (  \theta \xi) \right|^2 \rd \xi
			\\
			&= \theta^N 
			\int_{  \RN } \left\{ \left( 4\pi^2 \theta^{-2} |\xi|^2 + a \right)^s - a^s \right\} 
			\left| \scrF w_{R_0} ( \xi ) \right|^2 \rd \xi.
		\end{aligned}
	\]
By the monotone convergence theorem, 
	\[
		\lim_{\theta \to \infty} \int_{  \RN } \left\{ \left( 4\pi^2 \theta^{-2} |\xi|^2 + a \right)^s - a^s \right\} 
		\left| \scrF w_{R_0} ( \xi ) \right|^2 \rd \xi = 0.
	\]
Hence, for sufficiently large $\theta_0>0$, we have $I_{a} ( w_{R_0} (\theta_0^{-1} \cdot )  ) < 0$.

(ii) When $a=0$, the claim follows from the existence of $\ov{\gamma}_k$ with \eqref{eq:4.6} and 
the argument in \cite[section 3]{Amb-18-1}. 
Similarly, when $a>0$, the claim may be verified by \eqref{eq:4.6} for $G(t) = F(t) - a^s t^2 /2$ and 
the proof of \cite[Lemma 2.3 (iii)]{I-17}.

(iii) When $a=0$, the assertion is essentially proved in \cite[Theorem 7]{Amb-18-1}. 
Indeed, the argument in \cite[Theorem 7]{Amb-18-1} works even though we replace $\Hsr (\RN)$ 
by $H^s_{\cO_2}$ since $H^s_{\cO_2} \subset L^p(\RN)$ is compact for $2 < p < 2^\ast_s$.

	When $a>0$, for $\Hsr (\RN)$, the claim is shown in \cite[Propositions 3.1 and 3.2]{I-17}. 
For $H^s_{\cO_2}$, we need a slight modification in 
\cite[Steps 2 and 3 in Proposition 3.1 and Proposition 3.2]{I-17}. 
For Step 2, 
arguing as in \cite[Steps 1 and 2 in Proposition 3.1]{I-17}, we may prove that 
$w_0 \in H^1_{\cO_2}$ satisfies 
	\[
		f(w_0) - a^s w_0 \equiv 0 \quad {\rm in} \ \R^2.
	\]

	Our aim is to show $w_0 \equiv 0$. 
By (F1), there exists an $t_1>0$ such that 
$f(t)t -a^s t^2 < 0$ if $0< |t| < t_1$. Hence, 
$\cL^2( [ t_0 < |w_0| < t_1  ] ) = 0$ for all $t_0 \in (0,t_1)$ 
where $\cL^2$ denotes the $2$-dimensional Lebesgue measure. 
Writing $w_0^\ast$ for a Schwarz rearrangement of $w_0$, 
we obtain $\cL^2 ( [t_0 < |w_0| < t_1] ) = \cL^2 ( [t_0 < w_0^\ast < t_1 ] ) = 0$ 
for each $t_0 \in (0,t_1)$. 
Since $w_0^\ast \in H^1_{\rm rad} (\R^2) \subset C(\R^2 \setminus \{0\})$ 
and $w_0^\ast (x) \to 0$ as $|x| \to \infty$, we infer that $w_0^\ast \equiv 0$ and 
$w_0 \equiv 0$ in $\R^2$.

	For Step 3 and \cite[Proposition 3.2]{I-17}, 
we need to show 
	\[
		\int_{  \RN } \left( f(u_n) u_n - (1-\delta_0) a^s u_n^2 \right)_+ \rd x \to 0
	\]
for each $(u_n) \subset H^s_{\cO_2}$ with $u_n \rightharpoonup 0$ weakly in $H^s_{\cO_2}$. 
This assertion may be checked by noting 
the compact embedding $H^s_{\cO_2} \subset L^p(\RN) $ for $2 < p < 2^\ast_s$ and 
that 
$( f(t) t - (1-\delta_0) a^s t^2 )_+ \leq C_\e |t|^{p_0} + \e |t|^{2^\ast_s}$ 
for each $t \in \R$ and $\e>0$ holds due to (F1) and (f2') 
where $p_0 \in (2,2^\ast)$ and $\delta_0>0$ is a small number.
Thus,  we can verify that \PSPc holds when $a > 0$. 
	\end{proof}

Finally, we prove Theorem \ref{theorem:1.5} (iii) and (iv). 

	\begin{proof}[Proof of Theorem \ref{theorem:1.5} (iii) and (iv)]
(iii) By Lemma \ref{lemma:4.3} (i) and Theorem \ref{theorem:2.3}, 
we may find $(u_n) \subset H^s_{\cO_1}$ such that 
$I_a(u_n) \to c_{mp} > 0$ and 
$\| \rd J(0,u_n) \|_{(\R \times H^s_{\cO_1})^\ast} \to 0$. 
When $a = 0$, from the proof of \cite[Theorem 7]{Amb-18-1}, 
we may prove that $(u_n)$ is bounded in $H^s(\RN)$.

	On the other hand, when $a>0$, we again borrow the argument 
from \cite[Proposition 3.1]{I-17} for the boundedness of $(u_n)$ with modifications. 
First, as in \cite[Proposition 3.1]{I-17}, assume $\| u_n \|_{H^s(\RN)} \to \infty$ 
for contradiction. Then we may show that $\tau_n := \| u_n \|_{L^2(\RN)}^{-2/N} \to 0$ 
and $v_n(x) := u_n( \tau_n x)$ is bounded in $H^s(\RN)$. 
Our next aim is to show $v_n(x+x_n) \rightharpoonup 0$ weakly in $H^s(\RN)$ 
for any $(x_n)_n \subset \{ 0 \} \times \{ 0 \} \times \R^{N-2m_1}$. 
This can be verified by the arguments in \cite[Steps 1 and 2 in Proposition 3.1]{I-17} 
together with the modification in the proof of Lemma \ref{lemma:4.3} (iii).

	By noting $m_1 \geq 2$, $v_n \in H^s_{\cO_1}$ and 
$v_n(x+x_n) \rightharpoonup 0$ weakly in $H^s(\RN)$ 
for each $(x_n)_n \subset \{0\} \times \{0\} \times \R^{N-2m_1}$, 
it follows that 
	\[
		\sup_{ z \in \Z^N} \| v_n \|_{L^2(z+Q)} \to 0. 
	\]
Hence, $ \| v_n \|_{L^p(\RN)} \to 0$ for each $2<p<2^\ast_s$. 
Since 
$( f(t)t - (1-\delta_0) a^s t^2 )_+ \leq C_\e |t|^{p_0} + \e |t|^{2^\ast_s}$ for all $t\in \R$ 
due to (F1) and (f2') for sufficiently small $\delta_0 > 0$, we obtain a contradiction:
	\[
		\begin{aligned}
			0< \delta_0  \leq \tau_n^N \| u_n \|_{s,a}^2 - a^s (1-\delta_0) \| v_n \|_{L^2(\RN)}^2 
			&= o(1) + \int_{ \RN} \left[ f(v_n) v_n - a^s (1-\delta_0) v_n^2 \right] \rd x 
			\\
			& \leq o(1) + \int_{ \RN} \left[ f(v_n) v_n - a^s (1-\delta_0) v_n^2 \right]_+ \rd x \to 0.
		\end{aligned}
	\]
Thus $(u_n)_n$ is also bounded in $H^s(\RN)$ when $a > 0$.

	Now, if  $\sup_{z \in \Z^N} \| u_n \|_{L^2(z+Q)} \to 0$, then again we have 
$\| u_n \|_{L^p(\RN)} \to 0$ for each $2 < p < 2^\ast_s$. 
By a similar argument and $\rd_u J(0,u_n)u_n \to 0$, we observe that 
for sufficiently small $\delta_0 > 0$, 
	\[
		\| u_n \|_{s,a}^2 -  a^s (1-\delta_0) \| u_n \|_{L^2(\RN)}^2 
		\leq \int_{  \RN } \left[ f(u_n) u_n - a^s (1-\delta_0) u_n^2 \right]_+ \rd x  + o(1)
		\to 0.
	\]
This yields $I_a(u_n) \to 0 = c_{mp}$, which is a contradiction. 
The rest of the argument is identical to that of Theorem \ref{theorem:1.5} (i) and 
we omit the details. 

(iv) 
By Lemma \ref{lemma:4.3}, we may apply Theorem \ref{theorem:2.4} for $I_a$ in $H^s_{\cO_2}$ 
to obtain $(u_k)_k \subset H^s_{\cO_2} $ such that 
	\[
		\begin{aligned}
			& I_a(u_k) \to \infty, \quad 
			\rd I_a(u_k) = 0, 
			\\
			& 0=P_a(u_k) = \frac{N-2s}{2} \| u_k \|_{s,a}^2 
			+ as \int_{\RN} (a+4\pi^2|\xi|^2)^{s-1} |\scrF u_k(\xi)|^2 \rd \xi - N \int_{\RN} F(u_k) \rd x.
		\end{aligned}
	\]
Remark that by $I_a(u_k) \to \infty$, we also have $\| u_k \|_{H^s} \to \infty$ 
since $I_a$ is a bounded map. 
Thus, we complete the proof. 
	\end{proof}

\subsection*{Acknowledgement}
The author would like to express his gratitude to Kazunaga Tanaka 
on fruitful discussion of the abstract results in section \ref{section:2}. 
This work was supported by JSPS KAKENHI Grant Number JP16K17623 
and JP17H02851.

\end{document}